\documentclass[12pt,leqno]{amsart}
\usepackage{graphicx}
\usepackage{amssymb,amsfonts,amsthm}
\usepackage{enumerate}
\usepackage[T2B]{fontenc}
\usepackage[cp1251]{inputenc}

\newcounter{paranum}

\usepackage{enumitem}
\usepackage{amsmath,amsthm,amssymb}
\usepackage{amscd}
\usepackage{amsfonts}

\usepackage{xcolor}

\newcommand{\St}{{\mathrm{Stop}}}

\newtheorem{intro}{Theorem}

\newtheorem{pintro}[intro]{Proposition}

\newtheorem{theorem}{Theorem}[section]
\newtheorem{lemma}[theorem]{Lemma}

\newtheorem{cy}[theorem]{Corollary}
\newtheorem{prop}[theorem]{Proposition}

\theoremstyle{definition}
\newtheorem{df}[theorem]{Definition}
\newtheorem{ex}[theorem]{Example}

\newtheorem{rk}[theorem]{Remark}

\newcommand{\Z}{\mathbb Z}

\newcommand{\A}{\mathcal{A}}

\newcommand{\Comm}{\mathrm{Comm}\:}
\newcommand{\ZZ}{{\mathbb Z}}
\newcommand{\QQ}{{\mathbb Q}}

\newcommand{\vp}{{\varphi}}
\newcommand{\ld}{{\ldots}}
\newcommand{\ch}{\mathrm{char}}
\newcommand{\GL}{\mathrm{GL}}
\newcommand{\cL}{\mathcal{L}}
\newcommand{\cG}{\mathcal{G}}
\newcommand{\ad}{\mathrm{ad}}

\newcommand{\F}{{\mathbb F}}

\newcommand{\Q}{{\mathbb Q}}

\renewcommand{\Im}[1]{\mathrm{Im}\left( #1 \right)}

\newcommand{\chr}[1]{\mathrm{char}\left( #1 \right)}
\newcommand{\Aut}[1]{\mathrm{Aut}(#1)}

\newcommand{\cF}{\mathcal{F}}

\begin{document}
\title[Nilpotent algebras and quasigoups]{Nilpotent algebras, implicit function theorem, and polynomial quasigroups}
\author{Yuri Bahturin$^1$}\address{Department of Mathematics and Statistics\\Memorial University of Newfoundland\\St. John's, NL, A1C5S7} \author{Alexander Olshanskii$^2$}\address{Department of Mathematics\\1326 Stevenson Center\\Vanderbilt University\\Nashville, TN 37240} 
 \keywords{Equations in algebras, nilpotent algebras, quasigroups, Baker-Campbell-Hausdorf formula, commensurators, filiform Lie algebras, ordered algebras}
\subjclass[2020]{17B30, 	20N05, 06F25,17B01, 22E60}
\footnote{Supported by NSERC Discovery grant \# 227060-19}
\footnote{Supported by NSF grant DMS 1161294 and RFFR grant 11-01-00945 and Moscow Center of of Fundamental and Applied Mathematics, grant 075-15-2022-284}
\begin{abstract} We study finite-dimensional nonassociative algebras. We prove the implicit function theorem for such algebras. This allows us to establish a correspondence between such algebras and quasigroups, in the spirit of classical correspondence between divisible torsion-free nilpotent groups and rational nilpotent Lie algebras. We study the related questions of the commensurators of nilpotent groups, filiform Lie algebras of maximal solvability length and partially ordered algebras.
 \end{abstract}

 \maketitle

\tableofcontents

\section*{Introduction}\label{sINT} We consider nilpotent
algebras over fields. The algebras need not be associative or Lie. 

In the first part of the paper (sections \ref{part1}, \ref{sRA}), we observe that a finite system of equations over an algebra A has well defined Jacobian matrix $J$. If A is finite-dimensional and $J$ has maximal rank, then the implicit function theorem holds. The resulting polynomial function is  defined everywhere, not only in the neighborhood of zero.

\medskip

Polynomial functions define new operations on $A$, which we call (derived) \textit{polynomial operations}. Some classical examples of such
derived operations, turning $A$ into a group, include so called ``circle product'' $a\circ b = a+b+ab$, if $A$ is associative, and the product given by the Baker - Campbell - Hausdorff formula, if $A$ is a Lie algebra.

In the case where $A$ is a more general nonassociative algebra, any polynomial operation $\circ$ turns $A$ into a quasigroup or a loop. One of our main results in Part I of this paper, see Theorem \ref{pDerived}, states that  the original operations of addition and multiplication in $A$ can be, in a sense, restored from the derived operation $\circ$. So the results of Part I can be viewed as a far reaching generalization of the classical correspondence between torsion-free nilpotent groups and nilpotent Lie algebras. 

In the second part, dealing with the most popular case of the above connection, we discuss in depth three topics concerning the classical Malcev 
correspondence between rational nilpotent Lie algebras and divisible torsion-free nilpotent groups.

Firstly, we apply this connection to the computation of the group of \textit{commensurators} in finitely generated nilpotent groups. 
 
 Secondly, we look at the groups of polynomial mappings of nilpotent power-associative algebras. Their corresponding Lie algebras provide examples of nilpotent filiform Lie algebras of interest in Differential Geometry. 
 
 Lastly, in the concluding sections, we look at the the partial orderings of finite dimensional algebras. We show that in the case of nilpotent Lie algebras, they are closely connected to the orderings on their corresponding divisible torsion-free nilpotent groups. In the case of  general ordered algebras, we introduce a new invariant: the rank of the maximal partial order on an algebra. For solvable Lie algebra this rank is shown to be equal to the ordinary rank, that is, the dimension of a Cartan subalgebra. 
 
 The orderings on Lie algebras have been introduced and studied  by V. Kopytov \cite{KVM}. 

\section{Equations in nilpotent algebras and polynomial quasigroups}\label{part1}

\subsection{Polynomial functions on algebras and Jacobians}\label{sF}

An arbitrary algebra $A$ is a vector space over a field $\F$ with bilinear operation $A\times A\to A$. An absolutely free algebra $\cF(X)$ is the formal linear span of all so called \textit{nonassociative monomials} (with parentheses) in the \textit{alphabet} $X$, where the product of two monomials $u$ and $v$ is the monomial $(u)(v)$. Obvious parentheses are usually omitted, so instead of $((x_1)(x_2))(x_2)$ we write $(x_1x_2)x_2$.

For a nonassociative monomial $u=u(x_1,\dots,x_n)\in\cF(X)$, the \textit{degree} is defined as the length of its associative support, that is, the word in $X$ obtained by dropping the parentheses in $u$. So,  $\deg(x)=1$ for $x \in X$, and by induction, $\deg ((u)(v)) = \deg u + \deg v$.

Given a monomial $u =u(x_1,..,x_n)\in \cF(X)$ and the elements $a_1,..., a_n\in A$, the value $u(a_1,\dots,a_n)\in A$ of $u(x_1,..,x_n)$ is defined by an obvious induction on the degree $\deg(u)$, as soon as the value of $x_i$ is set to be $a_i$, $i=1,\dots,n$. Once the values of the monomials are defined on all $a_1, a_2,\dots\in A$, the value $f(a_1,\dots,a_n)$ of any polynomial  $f(x_1,\dots,x_n)\in{\ F}\{ X\}$ is well defined. Given an algebra $A$ over $\F$ and $a_1,\dots,a_n\in A$, the map $f(x_1,\dots,x_n)\mapsto f(a_1,\dots,a_n)$ is the unique homomorphism of algebras $\overline{\vp}:\cF(x_1,\dots,x_n)\to A$ extending the set map $\vp:x_1\mapsto a_1,\dots,x_n\mapsto a_n$. If $a_1,\dots,a_n$ generate $A$, then $\overline{\vp}$ is a homomorphism onto, and $A\cong\cF(x_1,\dots, x_n)/\ker\overline{\vp}$.

If the value of every monomial of degree $c+1$ in an algebra $A$ is zero, then
$A$ is called a \textit{nilpotent algebra of class $\le c$}. The linear span of
the values of monomials of degree $i$ in $A$ is denoted by $A^i$. So the
nilpotency class of $A$ is $c$ if $A^{c+1}=\{0\}$, but $A^c\ne \{0\}$.

We will also consider polynomial functions $f: A^m\to A$ with coefficients in $A$, defined by algebraic expressions. For example, $(x,y)\mapsto  a+ bx -2y+ (yc)x +xy$, where $a,b,c\in A$, is a function in two variables.

Formally speaking, a function $f$ is an element of the free product $\cF(X)\ast A$.
The latter free product is an algebra generated by the subalgebra $A$ and the subset $X$ such that the following holds. For any algebra $B$, it is true that any homomorphism $A\to B$ and any mapping
$X\to B$ uniquely extend to a homomorphism $\cF(X)\ast A\to B$. Such free products always exists for algebras over a field (See \cite{L}.) Moreover, one can take free products
in any variety $\mathcal{M}$ of algebras containing the algebra $A$, where a variety is a class
of algebras defined by a set of laws (e.g., the variety of associative algebras over
a field $\ F$). In this case, the free factor $\cF(X)$ is a free algebra in the variety.

Every function $f$ is the sum of the monomials involving coefficients from $A$
and an element $v$ from $\cF(X)$. Such presentation is not unique, but $v$ is a uniquely defined element of $\cF(X)$, because we have $f\mapsto v$ under the homomorphism
$A\to\{0\}$, $x\mapsto x$ for all $x\in X$. If the variety $\mathcal{M}$ contains a 1-dimensional
algebra with zero multiplication, then the coefficients $\lambda_i$ of the linear part $\lambda_1 x_1+\dots+\lambda_m x_m$ of $v$ are the uniquely defined elements of the field $\F$ (no matter whether $\F$ is infinite or not). To see this, one can choose in  $\mathcal{M}$ an algebra $B$ which is a vector space over $\F$ with trivial multiplication and $\{e_1,e_2,\ld\}$ and consider a homomorphism  $A\to \{0\}$,
$x_i\mapsto e_i$, $i=1,2,\ld$.

It follows that for a family of polynomial functions $f_1,..,f_r$ of $m$ variables on a nilpotent
non-zero algebra $A$ we have a well defined $r\times m$ matrix $J=[\alpha_{ij}]$, where
the element $\alpha_{ij}\in\F$ is the coefficient at $x_j$ in the linear part
of the function $f_i$. We will call this matrix the {\it Jacobian} matrix of the
set $(f_1,\ld,f_r)$.

\begin{rk} We use the nilpotency in the definition of $J$, since every nonzero
nilpotent algebra contains a 1-dimensional subalgebra, with zero multiplication. Otherwise, one
can meet impediments. For example, in the category of commutative algebras with
1, the multiplication by a nonzero scalar can coincide with the multiplication
by an element of the algebra $A$.
\end{rk}

A solution to a system
\begin{equation}\label{eSystem}
f_1(x_1,\ldots,x_m)=0,\ldots, f_r(x_1,\ldots,x_m)=0
\end{equation}
is an $m$-tuple $(x_1^0,\dots x_m^0)\in A^m$ such that $f_i(x_1^0,\dots,x_m^0)=0$
for $i=1,\ld, r$, i.e. every $f_i$ maps to zero under the homomorphism
$\cF(X)\ast A\to A$, which is identical on $A$ and substitutes $x_1\mapsto x_1^0,\dots x_m\mapsto x_m^0$.
\subsection{Nonsingular square systems of equations}\label{par1} A system (\ref{eSystem}) is called ``nonsingular square'' if $r=m$ and its
Jacobian matrix $J$ is nonsingular.

The following Proposition \ref{pNSQ} and its proof mimic A. L. Shmelkin's result in \cite{SHM} about the equations in the nilpotent groups. See \S\ref{par1}.

\begin{pintro}\label{pNSQ}
Every nonsingular square system of polynomial equations with coefficients in a nilpotent algebra $A$ has a unique solution in $A$.
\end{pintro}

\begin{proof}
Let $c$ denote the nilpotency class of $A$. If $c=1$, i.e. $A$ is equipped with zero
multiplication, then we have a nonsingular  square system of linear equations
\begin{equation}\label{eLinSystem}
\left\{\begin{array}{rcl}
\alpha_{11}x_1+\cdots+\alpha_{1m}x_m=u_1\\
\cdots\cdots\cdots\cdots\cdots\cdots\cdots\cdots\\
\alpha_{r1}x_1+\cdots +\alpha_{rm}x_m=u_m,
\end{array}
\right. 
\end{equation}
where $u_1,\dots,u_m$ are vectors from $A$ and $\alpha_{ij}\in\F$.
Since the unknowns $x_1,\dots,x_m$  are also vectors from $A$, the system (\ref{eLinSystem}) splits into $d$ systems ($d=\dim A\le\infty$) of ordinary  numerical systems each with the same nonsingular matrix of coefficients $[\alpha_{ij}]$. Thus, there is a (unique) solution to (\ref{eLinSystem}), hence to (\ref{eSystem}), if $c=1$.

If $c>1$, then by induction, there is a unique solution of (\ref{eSystem}) modulo $A^c$. Let $(x_1^0,\dots, x_m^0)$ be a preimage of the solution $\!\!\!\mod A^{(c)}$. Then $f_i(x_1^0,\dots, x_m^0)=c_i$, where $c_i\in A^c$, $i=1,\dots,r$.

Let $(y_1^0,\dots,y_m^0)$ be the solution in $A^c$ to the following system with write-hand sides in $A^c$:
\[\left\{
\begin{array}{rcl}
\alpha_{11}x_1+\dots +\alpha_{1m}x_m&=c_1\\
\dots\dots\dots\dots\dots\dots\dots\\
\alpha_{m1}x_1+\dots +\alpha_{sm}x_m&=c_m
\end{array}
\right. \]
Then $(x_1^0-y_1^0,\dots,x_m^0-y_m^0)$ will be a solution to (\ref{eSystem}). Indeed, since all $y_i^0\in A^c$ are annihilating, this follows because the substitution $x_i^0\mapsto x_i^0-y_i^0$, $i=1,\dots,m,$ only changes the linear part of each equation.

As for the uniqueness, it follows by a similar inductive argument. Assume there is uniqueness modulo $A^c$. Thus, the difference of two solutions consists of annihilating vectors. In this case, the difference satisfies a homogeneous system with nonsingular matrix $[\alpha_{ij}]$. It follows that the difference equals zero.
\end{proof}

\subsection{Implicit function theorem for nilpotent algebras}\label{ssIFT}

Now we consider a system of $r$ polynomial equations in $m$ variables
(\ref{eSystem}) under the assumption that $r\le m$ and that the Jacobian matrix
of (\ref{eSystem}) has rank $r$. Passing to the reduced row-echelon form of the Jacobian matrix, we may assume that the system has the form

\[\{x_i=F_i(x_1,\dots,x_m)\,|\,i=1,\dots,r\},\]
where $x_1,\dots, x_r$ are the pivot variables in the Jacobian part of the system.
It is possible that the pivot variables still enter the $F_i$'s, but only to the monomials of degree $>1$, with respect to the variables and the coefficients from $A$.

On the next step, we replace each occurrence of each pivot $x_j$ in each $F_i$, $i,j=1,\dots,r$, by $F_i(x_1,\dots,x_m)$. As a result, on the right side of the $i$th equation we will obtain a new polynomial $G_i(x_1,\dots,x_m)$, where the pivot $x_j$ enters only the monomials of degree $>2$. The left hand sides remain intact:
\[
\{x_i=G_i(x_1,\dots,x_m)\,|\,i=1,\dots,r\}
\]
Clearly, every solution of the first system will be a solution of the second one. The Jacobi matrix does not change. Repeating the same argument $c-1$ times, we obtain a system

\begin{equation}\label{eHi}
\{x_i=H_i(x_1,\dots,x_m)\,|\,i=1,\dots,r\},
\end{equation}
where the right hand sides do not have terms containing pivot variables, except for the terms of degree $c+1$.

As before, every solution of the original system is the solution of the latter one. By Proposition \ref{pNSQ}, both systems have a unique solution if the values of the free variables are fixed. Thus the solutions of the latter systems will be the solutions of the original system, meaning that the systems are equivalent. Therefore the following is true.

\begin{theorem}\label{tAG} If $A$ is a nilpotent algebra and the Jacobian matrix of the system (\ref{eSystem}) has maximal rank $r$, then one can express the solutions as polynomial functions of $m-r$ free variables defined on the entire algebra $A$.

 Moreover, if $A$ has finite dimension $d$, then the set of solutions of this system  is an affine space of dimension $d(m-r)$.
\end{theorem}

\begin{proof} The values of the right-hand sides of (\ref{eHi}) depend on the choice
of values for free variables in the $(m-r)$th direct power of $A$. So the set of solutions form a graph of a polynomial mapping defined on this power, which has
dimension $d(m-r)$ over the ground field $\F$.  The theorem is proved.
\end{proof}

The set of solutions can be found in an effective way, by subsequently finding solutions $\mod A^2, \mod A^3$, and so on, for non-singular square systems
or by iterative computation of functions $H_i$ (see (\ref{eHi})) if the Jacobian matrix has maximal rank.

\subsection{What if $A$ is not nilpotent?}

Being nilpotent algebra in Theorem \ref{tAG} is essential, as shown by the following. Recall that the \textit{multiplication algebra} $M(A)$ of an algebra $A$ is an (associative) subalgebra of the algebra of linear operators on the vector space $A$ generated by the left and right multiplications $x\mapsto ax$, $x\mapsto xa$, where $a,x\in A$.

\begin{theorem}\label{tEqnon_nil} For any non-nilpotent finite-dimensional algebra $A$, there is an 1-variable equation $f(x)=0$ with  Jacobian matrix $J=[1]$ but without solutions in $A$.
\end{theorem}

\begin{proof}
It is sufficient to produce an equation without solutions in a homomorphic image of $A$. Let us view $A$ as a left module over its multiplication algebra  $M(A)$ and consider a chief (Jordan-H\" older) series in  $A$. Since $A$ is not nilpotent, there is a factor $U/V$ where the action of  $M(A)$ is nontrivial. Without any loss of generality, we may assume that $V=\{ 0\}$ and $A/U$ is nilpotent (could be $\{ 0\})$.

There is $u\in U$ and $a\in A$ such that either $au$ or $ua$ is nonzero. Assume $au\ne 0$. Since $U$ is a simple $M(A)$-module, there is $\A\in M(A)$ such that $u=\A(au)$. The kernel of the linear map $g:U\to U$ given by $g(x)=x-\A(ax)$ contains $u$, hence nonzero. Then $\Im g\ne U$, and there is $v\in U$ such that for no $x\in U$ we have $g(x)=v$. Thus the equation $x-\A(ax)=v$ has no solutions in $U$. Note that $\A(ax)$ is
the sum of products involving both the variable $x$ and at least one factor $a$ from $A$.
So the function $\A(ax)-v$ has zero Jacobian matrix, while the Jacobian matrix of
 the function $f(x)= x-\A(ax)-v$ is $[1]$.

Suppose now there is a solution $y\in A$ of the equation $f(x)=0$ outside of $U$. Because $A/U$ is nilpotent, it follows that there is $k$ such that $y\in (A^k+U)\setminus (A^{k+1}+U)$. At the same time, both $\A(ay)$ and $v$ belong to $A^{k+1}+U$. Thus $y\ne \A(y)+v$ and $y$ is not a solution to our equation, as well.
\end{proof}
A simple example is the following.

\begin{ex}\label{e1}
Let $L$ be a two-dimensional Lie algebra $L=\langle e,f\;|\;ef=f\rangle$.
Then the equation
\[
x+xe=f
\]
has no solutions in  $L$.
\end{ex}

\subsection{Implicit functions and quasigroups.}\label{Ifq}

Recall the most-known example of an implicit function in (nilpotent)
algebras. Let $\F$ be a field of characteristic $0$ and $A(X)$ be the free
associative algebra over $\F$ with free basis $X=\{a_1,\dots, a_r\}$. Every element
of $A(X)$ is a unique linear combination of noncommutative monomial in $X$.
The algebra $L(X)$ generated by the set $X$ with respect to the commutator
operation $(u,v)=uv-vu$ is a free Lie algebra with basis $X$ (\cite{Bou}, 3.1, Theorem
1), and so $L(c,X)=L(X)/L^{c+1}(X)$ is a free nilpotent Lie algebra of nilpotency
class $c$. The exponent $\exp x=\sum_{i=0}^{\infty} x^i/i!$ has finitely many terms
in the algebra $A(X)$ factorized by all monomials of degree $\ge c+1$, and the equation
$\exp z = \exp x \exp y$ with Jacobian matrix $[-1,-1, 1]$ has a unique solution $z$ for arbitrary $x,y\in L(c,X)$. This is
given (by a finite version of) the Baker - Campbell - Hausdorff (BCH)
formula (\cite{Bou}, 6.4)
\begin{equation} \label{BCH}
z = \log(\exp x\exp y)= x+y +\frac 12 (x,y) + \frac{1}{12} (x,(x,y)) -\frac{1}{12} (y,(x,y))+\dots
\end{equation}
The mapping $(x,y)\mapsto z=x \circ y $ given by (\ref{BCH}) defines a group
operation $\circ$ on an arbitrary nilpotent Lie algebra $L$ (\cite{Bou}, 8.3).

A weaker generalization is valid for  the polynomial functions $\vp: A\times A\times A\to A$  in $3$ variables
defined on arbitrary nilpotent algebra $A$ over a field $\F$. Assume that the Jacobian matrix
$[\alpha, \beta, \gamma]$ of the function $\vp(x,y,z): A\times A\times A\to A$ has all non-zero entries. Then by Theorem \ref{tAG}, the equation $\vp(x,y,z)=0$
can be resolved with respect to each of the variables, e.g. $z = \lambda x +\mu y + f(x,y)$, where $\lambda,\mu \ne 0$ and the polynomial $f$ does not contain
the monomial $x$ and $y$ with coefficients from $\F$. Then the obvious change
of variables makes $\lambda=\mu=1$, and the function $\vp$
define the binary operation on $A$:
\begin{equation}\label{bino}
a\circ b = a +b + f(a,b)
\end{equation}
Let $Q(A)$ be the set $A$ with the operation $\circ$.
\begin{prop}\label{lQuasi} The operation $\circ$ on $A$ given by formula (\ref{bino}) makes $Q(A)$ a quasigroup, that is, each equation of the form $a\circ y = c$   and $x\circ b = c$ has a unique solution in $A$, for any $a,b, c\in A$.
\end{prop}

\proof Each of the equations $a\circ y = c$   and  $x\circ b = c$ has
a unique solution by Theorem \ref{tAG} since the Jacobian matrix
of the equation $z-x-y-f(x,y)$ is $[-1,-1,1]$, and so every
variable can be selected as the pivotal.
\endproof

\it{From now we will assume that the polynomial $f(x,y)$ defining the
operation $(\ref{bino})$ contains no constants from $A$, that is,
belongs to the free algebra $\cF(x,y)$.}

\begin{rk} \label{sub} This assumption immediately implies that every subalgebra $A'\subset A$ becomes a subquasigroup of $Q(A)$, because $\circ$ is a {\it derived operation}, i.e. the composition of the defining operations on $A$. Also, any algebra homomorphism $A\to B$ is also a quasigroup homomorphism
$Q(A)\to Q(B)$.
\end{rk}

\begin{prop}\label{loo}
If the polynomial $f(x,y)$ from (\ref{bino}) has no monomials depending on one variable only, $Q(A)$ is a loop in the sense that it has a neutral element $0$.
\end{prop}

\proof Indeed, under the assumption of the statement, $f(a,0)=f(0,a) =0$
for every $a\in A$, i.e. $a\circ 0=0\circ a =a$ by (\ref{bino}).\endproof

For the detailed treatment of quasigroups and loops see \cite{BRU}.

An algebra (a quasigroup, loop, group,...) A is called {\it relatively
free} or free in a variety $\mathcal{V}$ of algebras (quasigroups,...) if it
belongs to $\mathcal{V}$ and has a set of generators $X$ ({\it free basis})
such that every mapping $X\to B$ to arbitrary $B\in \mathcal{V}$ extends
to a homomorphism $A\to B$. It follows from Birkhoff's theorem that $A$
is relatively free iff every mapping $X\to A$ extends to a homomorphism
$A\to A$.

\begin{prop}\label{r1}
Suppose  a nilpotent algebra $A$ over a field $\F$ is relatively free with a free basis $X$, then the subquasigroup $Q(X)$ generated by $X$ in the quasigroup (loop) $Q(A)$) defined by (\ref{bino}), is relatively free, too.
\end{prop}

\begin{proof} Since any mapping $\alpha:X\to A$ extends to a homomorphism  of algebras $\overline{\alpha}:A\to A$ and $\overline{\alpha}$ is also a quasigroup homomorphism $Q(A)\to Q(A)$ by Remark \ref{sub}, every mapping $X\to Q(A)$ (in particular, $X\to Q(X)$) extends to a quasigroup homomorphism, which completes the proof.
\end{proof}

{\bf Examples}. (1) If  $A$ is a free nilpotent Lie algebra of nilpotency class $c$ with a basis $X$, over the field of rational numbers, then $Q(X)$ is a free nilpotent group of class $c$. For
other nilpotent varieties of algebras Lie, the group varieties obtained by (\ref{BCH})  are described in \cite[Chapter 8]{Ba}.

(2) The same formula (\ref{bino}) sets a correspondence between the real (nilpotent) Malcev algebras and analytic Moufang loops \cite{Ku}.

(3) Mostovoy, Shestakov and Perez-Izquierdo ( \cite{MSPI1}) prove that in the  free nilpotent, nonassociative algebra with free generators $x,y$, the loop generated by $x,y$ with respect to the operation $x\ast y=x+y+xy$ is a free nilpotent loop.

\medskip
\begin{rk}\label{sssMO}
Given a set $S$ and natural numbers $m,n$, one can consider multiple operations
\begin{equation}\label{eMO}
\omega: \underbrace{S\times\cdots\times S}_m\to\underbrace{S\times\cdots\times S}_n.
\end{equation}
In the case where $S=A$ is a nilpotent algebra, such operations appear while solving systems of polynomial equations (\ref{eSystem}) with Jacobian matrix of
maximal rank. Namely, given such a system, according to Theorem \ref{tAG}, we can write the solutions to (\ref{eSystem}) in the form (\ref{eHi}). Giving  $x_{r+1},\dots,x_m$ values in $S$, we compute the values for $x_1,\dots,x_r$ in $S$, using (\ref{eHi}). This provides us with an operation $\underbrace{A\times\cdots\times A}_{m-r}\to\underbrace{A\times\cdots\times A}_r$. If $m=2$, $r=1$ (or $m=2r$) and there are two disjoint non-singular
$1\times 1$ (resp., $r\times r$) submatrices in $J$, then we have two mutual inverse polynomial mappings $A\to A$ (resp., $r$-th power of $A$ to itself). In the case $m=3$, $r=1$, we obtain an operation $A\times A\to A$, considered above. An important case of polynomial mappings $m=2,\,r=1$ is treated in Section \ref{GU}. 

Thus, we get a large family of quasigroups and their
generalizations depending on the choice of $A$ and the choice of polynomial
operations. If $\F$ is a finite field, they might be interesting for
the Cryptography (see, for instance, \cite{G}, \cite{CGV}).
\end{rk}

\section{Reconstruction of algebras}\label{sRA}
\subsection{Quadratic part of the operation $\circ$.}

The function $f(x,y)$ from (\ref{bino}) has quadratic part $kxy+lyx+mx^2+ny^2$, where  $k,l,m,n,\in{\F}$. In this section, we want to simplify this quadratic part, that is, to derive a new operation with a simpler quadratic part, using $\circ$ and the multiplication by scalars from $\F$. This technical result is used in the proof of the main Theorem \ref{pDerived}
in the next section. Thus, $A$ is a nilpotent algebra of class $c\ge 2$, and we have
\begin{equation}\label{eStar}
a\circ b =a+b+kab+lba+ma^2+nb^2+...,
\end{equation}
where the dots here and everywhere in this section are used for the monomials of degree $\ge 3$.

{\it We will assume in this section that the ground field $\F$ has at least $3$
elements, $k\ne l$ and if $m=n=0$, then $k\ne- l$.}

\begin{lemma}\label{la2} Under the above hypotheses, there is a derived  operation
$a\#b=a^2+\dots$ for $\circ$ and the scalar multiplications.
\end{lemma}
\begin{proof}
At first, we want to find a derived operation for (\ref{eStar}), where  one of $m$, $n$ is nonzero. Since we can multiply by any scalars, any $a*b=\lambda a$, $\lambda\in{\F}$, is a derived operation. Now if $a\circ b=a+b+kab+lba+\dots$, then
the following derived operation for $\circ$:
\[
a\# b=(a\circ b)\circ(-a)=b+(k-l)ab+(l-k)ba-(k+l)a^2+\dots
\]
has $k+l\ne 0$ as a coefficient of $a^2$.

So we may assume that $m\ne 0$, the case $n\ne 0$ being similar.
We set $u_0=a$ and $u_{s+1}=u_s\circ 0$, $s=1,2,\dots$. Then $u_s=a+sma^2+\dots$. Indeed, by induction,
\[u_{s+1}=u_s\circ 0=(a+sma^2+\dots)\circ 0\] \[=a+sma^2+0+m(a+sa^2)^2+\dots=a+(s+1)ma^2+\dots,\]
as claimed.

Then consider
\[v_s=u_s\circ(-a)=(a+sma^2+\dots)\circ(-a)\] \[=a+sma^2-a+(-k-l+m+n)a^2+\dots
=((s+1)m-k-l+n)a^2+\dots\]
Since $m\ne 0$, one of $(sm-k-l+n)$ or $((s+1)m-k-l+n)$ is nonzero. For instance, one of  $(-k-l+n)a^2+\dots$ or $(m-k-l+n)a^2+\dots $ is a derived operation. Thus, multiplying by a nonzero scalar, we can see that $a^2+\dots$ is a derived operation.
\end{proof}

\begin{lemma}\label{nom} There is a derived operation
\begin{equation}
a\ast b=a+b+kab+lba+nb^2+\dots, \mbox{ where }k\ne\pm l\label{gtm4}
\end{equation}
\end{lemma}
\begin{proof}
Lemma \ref{la2} and multiplication by $t\in{\F}$ provide us
with a derived operation $ta^2+\dots$. Now
\[(a\circ b)\circ(t a^2+\dots)=(a+b+kab+lba+ma^2+nb^2)+ta^2\] \[+m(a+b+kab+lba+ma^2+nb^2)^2+\dots\]
\begin{equation} \label{gtm2}
=a+b+(k+m)ab+(l+m)ba+(2m+t)a^2+(n+m)b^2 +\dots
\end{equation}

If $\chr{\F}=2$, then choosing $t=0$ we get rid of $m$ and  still have $k+m\ne\pm (l+m)$.

Now assume  $\chr{\F}\ne 2$. If $t=-2m$, then setting $k'=k+m$, $l'=l+m$ and $n'=n+m$ we turn  (\ref{gtm2}) into
\[a * b=(a\circ b)\circ(-2m a^2+\dots)=a+b+k'ab+l'ba+n'b^2+\dots\]
Although $k'\ne l'$, we might have $k'=-l'$. So we treat this case separately. Let us assume $k+m=-l-m$ and choose $t=-m$ in (\ref{gtm2}). Then we get
\[a\$ b=(a\circ b)\circ(-m a^2+\dots)=a+b+k'ab+l'ba+ma^2+n'b^2+\dots\]
Next we create a desired operation $a * b=(a\$ b)\circ(-2m a^2+\dots)$. In this case,
\[a * b=(a\$ b)\circ(-2m a^2+\dots)=a+b+k'ab+l'ba+ma^2+n'b^2-2ma^2\] \[+m(a+b+k'ab+l'ba+ma^2+n'b^2)^2+\dots
=a+b+k''ab+l''ba+n''b^2+\dots,\]
where $k''=k+2m=-l-m+2m=-l+m$ and $-l''=-l-2m$. Since we assume $m\ne 0$, we have that $k''\ne -l''$. We still have $k''\ne l''$, and so the lemma is proved.
\end{proof}

\begin{lemma}\label{lakl} There are derived operations $a\$b=a+b+kab+lab+\dots$, where
$k\ne\pm l$, and $a\#b =kab +lba+\dots $ with $k\ne\pm l$.
\end{lemma}
\begin{proof}
Lemmas \ref{nom} and \ref{la2} provide us with the following operation
\[(a* b)\circ(-nb^2+\dots)=a+b+kab+lba+nb^2-nb^2+\dots=a+b+kab+lba+\dots\]

Now we can get rid of the linear part, as follows. First, we assume that $\ F$ has more than three elements. We pick $0\ne \mu\in{\ F}$ and consider the derived operations
\[(-\mu a)\$(-\mu b)=-\mu a-\mu b+\mu^2 kab+\mu^2 lba+\dots,\]
\[\mu(a\$ b)=\mu a+\mu b+\mu kab+\mu lba+\dots\]
and their derived operation
\[((-\mu a)\$(-\mu b))\$(\mu(a\$ b))\] \[=(\mu+\mu^2) kab+(\mu +\mu^2)lba-(k+l)\mu^2(a^2+ab+ba+b^2)+\dots\]
\[=(\mu k-\mu^2 l)ab+(\mu l-\mu^2k)ba-(k+l)\mu^2a^2-(k+l)\mu^2b^2+\dots,\]
that after division by $\mu$, has the form
\[a\vert b=(k-\mu l)ab+(l-\mu k)ba-(k+l)\mu a^2-(k+l)\mu b^2+\dots \]
It now follows by Lemma \ref{la2}  that we can form a derived operation
\[((a\vert b)\vert (k+l)\mu a^2+\dots)\vert (k+l)\mu b^2+\dots)=(k-\mu l)ab+(l-\mu k)ba=\hat{k}ab+\hat{l}ba+\dots,\]
where $\hat{k}=(k-\mu l)$ and $\hat{l}=(l-\mu k)$. Now, considering $k\ne\pm l$, we have that $\hat{k}=\hat{l}$ implies $\mu=-1$ while $\hat{k}=-\hat{l}$ implies $\mu=1$. So if we take any $\mu\ne\pm 1$, we arrive at a derived operation of the form $\#$, as in the formulation of the lemma.

In the remaining case ${\F}=\Z_3$ (or just if $\chr{\F}\ne 3$), we proceed as follows. Applying the operation $\$$, we have $(-a)\$(-b)=-a-b+kab+lba+\dots$ and
\[
(a\$b)\$((-a)\$(-b))=-(k+l)(a+b)^2+2kab+2lba+\dots.
\]
Now since  $(a\$ b)^2 = (a+b)^2+\dots$, we obtain
\[
[(-(k+l)(a+b)^2 + 2k ab + 2 lba+\dots)]\$ [(k+l)(a\$ b)^2+\dots] = 2k ab +2l ba+\dots.
\]
Here $(a\$ b)^2+\dots$ is a derived operation provided by Lemma \ref{la2}, and so is $2k ab +2l ba+\dots.$
After dividing by 2, we have proved that in all cases, we have a desired derived operation $\#$.
\end{proof}

\subsection{From (quasi)groups back to algebras}

In this section, we show that the addition and multiplication in an
nilpotent algebra can be restored from the quasigroup operation $\circ$ (\ref{eStar}) and scalar multiplication.

\begin{theorem}\label{pDerived}
Let $\ F$ be a field with at least $3$ elements. Then for any nilpotent of class $c\ge 2$  algebra $A$, the operations $a+b$ and $ab$ are derived from the $\circ$-operation (\ref{eStar}) and the scalar multiplications by the elements of $\F$, if and only if $k\ne l$ in all cases and $k\ne-l$ if $m=n=0$.
\end{theorem}

\begin{proof} Let us start with the part ``only if''. In case $k=l$, even for
$c=2$, one can take a nilpotent Lie algebra $A$ and notice that $\circ$ is
just the addition in $A$, and thus no nonzero product can be obtained from
$a$ and $b$ by repeated application of $\circ$ and scalar multiplications.

If $k=-l$, while $m=n=0$, the same kind of contradiction appears when we consider an image which is a  nilpotent commutative algebra of class $c=2$.

To prove part ``if'', one may assume that $A$ is a free nonassociative
nilpotent algebra $\mathcal{N}$ of class $c\ge 2$, because every nilpotent algebra $A$ is a homomorphic image of a free nilpotent algebra.

If $c=2$, then changing $a$ and $b$ places, we have a derived operation $a\& b=lab+kba$ from the operation $\#$ given by Lemma \ref{lakl}. The linear combination of these two operations with arbitrary coefficients $r,s$ is the derived operation $(r(a\#b)\circ(s(a\& b))$. Since  $k\ne\pm l$, the coefficients can be chosen so that we simply have the derived operation $ab$.

Now using the operations $\$$ from Lemma \ref{lakl} and $ab$ from the previous
paragraph, we get
\[(a\$b)\$(-kab)=a+b+kab+lba-kab=a+b+lba.\]
The summand $lba$ can be removed in a similar way, and so we obtain the derived operation $a+b$. Thus we have finished the case $c=2$.

As a result, we have proved  that any element of the free nilpotent algebra of class 2 can be expressed in terms of $\circ$ and the scalar multiplications. We will proceed by induction on $c$ in order to prove that the same claim holds for any $c>2$.

We consider a monomial $w=uv\in {\mathcal{N}}^{c}$. By induction, the factor $u$ can be expressed in terms of $\circ$ and the scalar multiplications  modulo ${\mathcal{N}}^{c}$, that is, there is $u'\in\mathcal{N}^{c}$ such that $u+u'$ can be expressed in the desired form. Similarly, there is $v'\in{\mathcal{N}}^{c}$ such that $v+v'$ is expressible. We will plug these expressions in the derived formula for the expression of $ab$ modulo ${\mathcal{N}}^{c}$, which has the form $ab+g(a,b)$, where $g(a,b)\in {\mathcal{N}}^{c}$.

Let us write $ g(a,b) = g_1(a,b) + g_2(a) +g_3(b)$,
where all monomials in $g(a,b)$ contain both $a$ and $b$, $g(a)$ only $a$ and $g(b)$ only $b$.
Substituting 0 for $b$ in $ab+g(a,b)$, we will obtain two derived operations $g_2(a)$ and $-g_2(a)$. Combining them with $ab+g(a,b)$ by means of the operation $\#$ from
Lemma 3 and using that $g$-terms belong to
${\mathcal{N}}^c$, we get
\[
(ab + g(a,b))\#(-g_2(a)) = ab+g_1(a,b) +g_3(b).
\]
in $\mathcal{N}$. In a similar way, we remove  $g_3(b)$. As a result, we obtain a derived operation $ab + g_1(a,b)$. Returning to our previous notation, we now have a derived operation $ab + g(a,b)$, where each monomial of $g(a,b)$ contains both $a$ and $b$ to a nonzero degree.

Then we will get in $\mathcal{N}$ the  equality
\[
(u+u')(v+v')+g(u+u',v+v')=uv.
\]
Indeed, $(u+u')(v+v')=uv$, because $u'$ and $v'$ belong to ${\mathcal{N}}^c$. Now since $c>2$, one of $u+u'$ or $v+v'$ is in ${\mathcal{N}}^2$ and $g(a,b)$ is an annihilating polynomial. So $g(u+u',v+v')=0$. As a result, $uv$ is expressible in terms of $\circ$ and the multiplication by the scalars.

So far we have seen that every monomial from ${\mathcal{N}}^c$ is expressible, but then also every polynomial (element) in $\cF^c$ is expressible since for monomials of degree $c$, the $\circ$ product is just the summation.

By induction, we already had a derived operation of the form $a\&b=a+b+g(a,b)$, where $g(a,b)\in {\mathcal{N}}^c$. For any $h\in {\mathcal{N}}^c$, we would have
\[
(a\&b)\&h=a+b+g(a,b)+h +g(a+b,h)=a+b+g(a,b)+h+g(a+b,0),
\]
following since $h$ is annihilating in $\mathcal{N}$ and $g\in {\mathcal{N}}^2$. Choosing 
\[
h=-g(a,b)-g(a+b,0),
\]
we will have $(a\,\&\,b)\,\&\,h=a+b$ in $\mathcal{N}$. Quite similarly, since by induction there is a derived operation $\bullet$ such that $a\bullet b=ab+g(a,b)$, we can write $(a\bullet b)\bullet(-g(a,b))=ab$, for nilpotent algebras of class $c$. So both $a+b$ and $ab$ are derived operations for $\circ$ and the multiplication by scalars in the nilpotent of class $c$  algebra $\mathcal{N}$.
\end{proof}

{\bf Example.} If  $\F$ has cardinality $2$, the conclusion of Theorem \ref{pDerived} is not true.  Let us look at a commutative nilpotent 3-dimensional algebra $A$ over $\Z/2\Z$ given by
\[
A=\langle a,b|\, a^2=b^2=0\rangle,
\]
Any operation (\ref{eStar}) satisfying the hypotheses of Theorem \ref{pDerived} becomes $a\circ b=a+b+ab$. It is easy to check that with respect to this operation, $a$ and $b$ generate the Klein's Viergruppe $\{0,a,b,a+b+ab\}$, the multiplication by $0$ and $1$ does not change this. So we cannot get $a\%b=ab$ as  a derived operation (and cannot recover the whole algebra, consisting of 8 elements).

\subsection{Operations defined by power series}\label{ps}

The BCH formula (\ref{BCH}) is given by an infinite power series, which, provided $\ch (\F) = 0$, can
be applied to a nilpotent Lie algebra $L$ of arbitrary class $c$ if one trims
it to the Lie polynomial of degree $c$.
Similarly one can treat formula (\ref{bino}) as a power series (neither associative nor Lie in general) and apply its trimmed version $\circ_c$ to an arbitrary nilpotent algebra of class $\le c$.

Assume now that $A=\bigoplus_{i=1}^{\infty} A_i$ is a graded algebra with $(A_i)(A_j)\subset A_{i+j}$ for $i,j =1,2, \dots$. For example, arbitrary
relatively free algebra $A=\cF(X)$ over an infinite field $\F$ is graded by degrees of monomials in $X$; see \cite{Ba}, Theorem 4.2.4.

So under the assumptions of Theorem \ref{pDerived}, we are able to obtain the
basic operations $+_c$ and $\times_c$ as derived from $\circ$ and the scalar multiplications in every nilpotent factor algebra $A/A^{c+1}$. When one changes $c$, these formulas interact in the same way as in the BCH-formula (\ref{BCH}), that is, are truncations of a single infinite series. More precisely, the following is true.

\begin{prop} \label{conn} Under the assumptions of Theorem \ref{pDerived}, $+_{c}$ and $\times_{c}$ coincide as derived operations with $+_{c-1}$ and $\times_{c-1}$, respectively, on nilpotent algebras of class $c-1\ge 1$. Moreover, there exist a natural number 
$k=k(c)$ and the operations $\#_1,\dots\#_k, \bullet_1,\dots,\bullet_k$, derived from $+_{c}$ and the scalar
multiplications, such that
\begin{equation}\label{am}
a +_{c}b = (((a+_{c-1} b)\#_1(a \bullet_1 b))\#_2\dots )\#_k(a \bullet_k b),
\end{equation}
where each application of $\#_i(a \bullet_i b)$ does not change the value of
the preceding prefix of this formula in the nilpotent algebras of class $\le c-1$.
The same statement (but with different auxiliary operations $\#_1,\dots, \bullet_l)$ relates the operations  $\times_{c}$ and $\times_{c-1}$.
\end{prop}

\begin{proof} The statement directly follows from the proof of Theorem \ref{pDerived}. For instance (see the last paragraph there),
$a +_{c} b = (a+_{c-1} b)+_{c-1} h$, where $h$ is a polynomial of degree $c$
derived from $\circ_c$ and the scalar multiplications.
\end{proof}

Proposition \ref{loo} defines the quasigroup $Q(A)$ for power series (\ref{bino}) and every graded algebra $A$, but the elements of $Q(A)$ are power series in this case. Proposition \ref{conn} shows that the basic operations in the algebra of power series can be restored, but the formula (\ref{am}) must
be extended by the obvious induction on $c$ to an infinite one.

\subsection{Anticommutative algebras}\label{ssRW}

Assume now that the operation (\ref{bino}) is defined in the variety $\mathcal{V}$
of anti-commutative algebras given by the identity $x^2=0$ or in a
subvariety of $\mathcal{V}$ (for example in the variety of Lie algebras), over a field $\F$. Then $a\circ 0 = 0\circ a =a $ in a nilpotent algebra $A\in \mathcal{V}$, i.e. zero element of $A$ is neutral in the quasigroup $Q(A)$. Furthermore, $a\circ (-a)= (-a)\circ a =0$, that is, there is a two-sided inverse element for every element of $Q(A)$.
In other words, $Q(A)$ is a {\it loop} with respect to the operation $\circ$.

It is obvious that $na\circ ma = (n+m)a$ for arbitrary $n, m \in \F$, i.e. every element of $Q(A)$ is contained in a subgroup (with the same neutral element) isomorphic to the additive group of $\F$. Thus,
$Q(A)$ is a power-associative loop.

\begin{df} Let $L$ be a power-associative loop. We call $L$ \emph{divisible} if for any $a\in L$ and any natural number $n$ there is $x\in L$ such that $x^n=a$. If such $x$ is unique, we call $L$ \emph{uniquely divisible}.
\end{df}

Obviously, for $\alpha\in \F \backslash\{0\}$, the equation $\alpha x = a$ has a unique solution $\alpha^{-1}a$. Therefore if $\ch\; \F =0$, then $L=Q(A)$ is a (uniquely) divisible loop. Passing to the multiplicative notation,
we have a unique $n$-th root of any $a\in L$ for every integer $n\ne 0$. In particular, $L$ is a torsion free loop.

If $\F=\mathbb Q$, it follows that the power $a^q$ is well defined for rational $q=s/t$, and in additive notation, it is equal to  $qa\in A$. However $a^q$
is a solution of the equation $x^t = a^s$. Therefore the multiplication by scalars in $A$ can be expressed in pure loop terms, and Theorem \ref{pDerived}
implies

\begin{theorem}\label{overQ} Let $A$ be a rational nilpotent of class $\ge 2$ anticommutative algebra. Let $Q(A)$ be the loop obtained from $A$ by the circle operation (\ref{eStar}), where $k\ne l$. Then $Q(A)$ is a uniquely divisible nilpotent loop.  

The original operations  $a+b$ and $ab$ can be uniquely reconstructed as derived operations from the circle operation alone.
\end{theorem}

\begin{proof}
If $A$ is nilpotent of class $c$, then we can see from (\ref{bino}) that every
element $z\in A^c$ commutes in the loop $L$  with every element $a\in L$.
We also have $z\circ(a\circ b)=(z\circ a)\circ b=a\circ (z\circ b)$ for $a,b\in L$. Therefore $L_1=A^c$ is a central subgroup in $L$. Hence it is normal in $L$,
and the factor loop $L/L_1$ is well defined and consists of cosets modulo
$A^c$ (see \cite{BRU}). By induction, one obtains an ascending central series
${0}\le L_1\le L_2\le\dots \le L_c =L$, and so $L$ is a nilpotent loop. 

The proof of the second claim follows from the remarks preceding the theorem.
\end{proof}

The most important and well-known example of the correspondence of the type $A\leftrightarrow Q(A)$ was found by A.I. Mal'cev as the correspondence between rational nilpotent Lie algebras
and  divisible torsion free nilpotent groups in \cite{M}. Now it can be
formulated as follows.

\begin{theorem}\label{mal} \emph{[Malcev Correspondence]} For every rational nilpotent Lie algebra $A$, the operation $\circ$ defined by (\ref{BCH}) converts $A$ into  a divisible torsion free nilpotent group $\cG(A)$, of the same nilpotency class. At the same time, one can convert any torsion free nilpotent divisible group $G$, into a rational nilpotent Lie algebra $\cL(G)$, where the scalar multiplication is $qa = a^q$ ($q\in \mathbb Q$) and the addition and multiplication are the derived operations, obtained in Theorem \ref{overQ}. Moreover, for a group $G$ converted into a Lie algebra $\cL(G)$, we have $\cG(\cL(G))=G$ and for a Lie algebra $A$ converted into a group $\cG(A)$, we have $\cL(\cG(A))=A$. In both cases, not only the sets but also the operations are the same.$\;\Box$
\end{theorem}

The original formulation did not say about the way of reconstructing $A$
from $\cG(A)$. But it was not a big secret, and the formula $x+y = xy[x,y]^{-1/2}\dots$ can be already seen  in \cite{Laz,Gl,St}. Here $[x,y]$ is the commutator
in a nilpotent torsion free, divisible group $G$, the addition is the
operation in $L(G)$, and $\dots$ stand for the product of rational powers of longer group commutators. A number of further factors in the formulas for $+$ and $\times$
in the nilpotent Lie algebra $L(G)$ were computed in \cite{CdGV}.  

An important paper  is \cite{Laz}. In this paper, the author deals also with the case of $p$-groups and Lie algebras over the fields of positive characteristic.  A modern presentation of BCH-operation and the inverse formulas
is given in \cite{KHU}. For the convenience of the reader,  we will give more precise statement of these and other properties of Malcev correspondence in an auxiliary section \ref{sFDMC}.

\subsection{More attention to polynomial quasigroups!}

The quasigroup $Q(A)$ depends on the choice of the nilpotent algebra $A$ in some variety of algebras over a field $\F$ and on the polynomial operation (\ref{eStar}). Even under assumptions of Theorem \ref{pDerived}, different polynomial operations $\circ$ and $*$ on the same algebra $A$ can define isomorphic quasigroups
$Q(A,\circ)$ and $Q(A,\star)$.

For example, consider a free anti-commutative  algebra $A$
of nilpotency class 2 with $r$ free generators $a_1,\dots, a_r$
and two operations
\[
x\circ y = x+y +xy \mbox{ and }\;x\star y = x+y +k xy,\mbox{ where }k\in \F\backslash\{0\}.
\]
It is easy to check that both $Q(A,\circ)$ and $Q(A,\star)$ are
divisible nilpotent groups of class two, and they are isomorphic
under the following mapping defined on $A$, where the coefficients $\lambda_i$ and $\mu_{ij}$ are arbitrary elements of $A$:
\[
 \sum_{i=1}^r \lambda_i a_i+\sum_{1\le i<j\le r}\mu_{ij} a_ia_j\mapsto \sum_{i=1}^r \lambda_i a_i+k \sum_{1\le i<j\le r}\mu_{ij} a_ia_j.
 \]

Similarly one can modify the BCH formula (\ref{BCH}),
multiplying every homogeneous summand of degree $d$
by $k^{d-1}$. In general, if $\vp:Q(A,\circ)\to Q(A,\star)$
is a quasigroup isomorphism, i.e., $\vp(x)\star \vp(y)=\vp(x\circ y)$, then we have $x\circ y = \vp^{-1}(\vp(x)\star\vp(y))$. Therefore choosing an appropriate bijections $\vp$ on $A$ and applying the last formula to an operation $\star$, one can produce new operations with isomorphic quasigroup $Q(A)$.

Groups are too good to occur frequently among polynomial quasigroups $Q(A)$. However other polynomial quasigroups and loops appear in the literature very seldom in comparison with
groups and their correspondence to Lie algebras mentioned in Theorem \ref{mal}. Just few examples of this kind are mentioned in Subsection \ref{Ifq}. Let us consider one more rather particular example.

Let $\circ$ be the operation $a\circ b = a+b+ab$ on a
nilpotent Lie algebra $A$. It is easy to see that $Q(A)$ is now a power-associative nilpotent loop. Moreover, one can easily check the law $x\circ (y\circ x)=(x\circ y)\circ x$ in this loop.
Further computations give a commutator identity. Namely, denote
by $[a,b]_{\circ}$ the commutator $(a\circ b)\circ(a^{-1}\circ b^{-1})$. (Here we use multiplicative notation for the loop $Q(A)$.) Then the commutator $[[x,y]_{\circ},[x^{-1},y^{-1}]_{\circ}]_{\circ}$ is always the neutral element of this loop. Thus we obtain an identity which is not a consequence of the associativity.

Here we do not obtain neither the defining laws nor the structures for the class of polynomial loops arising due to this particular correspondence. Our goal is just to convince the reader
that polynomial quasigroups and loops form a large and interesting class of algebraic structures connected to algebras They deserve to be studied closely, and there is a hope that many important
subclasses of nilpotent algebras and polynomial operations
will be found, where the theory generalizes in the spirit of what
is known for the BCH-correspondence between nilpotent (or graded, see subsection \ref{ps}) Lie algebras and groups.


In this part, we still try to work with the general nilpotent algebras, wherever possible. At the same time, the applications mostly deal with Lie algebras and nilpotent groups. For the reader's convenience, we provide the details of now classical results on divisible groups and Malcev correspondence, which we briefly discussed in Section \ref{ssRW}.

\section{Applications and related topics}\label{sFDMC}

In the next auxiliary section we provide some known facts for the reader's convenience.

\subsection{Divisible hulls of  nilpotent groups}\label{ssDNG}
To apply Malcev Correspondence to arbitrary finitely generated nilpotent groups,
we first recall some known facts about the embedding of torsion free nilpotent
groups in divisible nilpotent groups.

A subgroup $B$ of a group $A$ is called \textit{isolated} if for any natural $n$ and any $a\in A$ such that $g^n\in B$ it follows that $a\in B$. 

A (multiplicative) torsion free group $H$ is called a {\it divisible hull} of a group $G$ if $H$ is a divisible group, containing $G$ as a subgroup and for every $h\in H$, there is $n>0$ such that $h^n\in G$.
Typical examples are ${\mathbb Q}^n$ and the group of upper unitriangular
matrices $UT(n, {\mathbb Q})$ as divisible hulls for the additive group
${\mathbb Z}^n$ and the group $UT(n, {\mathbb Z}) $, respectively.

The following theorem due to A. I. Malcev \cite{M} is now classical. In the  book \cite[\S 67]{K53} A. G. Kurosh calls Malcev's Theorem ``the main theorem in the whole theory of torsion-free nilpotent groups''.   The original topological treatment in \cite{M} was later complemented by the algebraic approaches in \cite{Laz} and also in \cite{SH}. For an up-to-date treatment see \cite[\S 9.3]{KHU}

\begin{lemma}\label{hull}
For every torsion free nilpotent group $G$, there is a
divisible hull, denoted $\sqrt G$. It is unique, in the sense that every isomorphism of torsion free nilpotent groups $G_1\to G_2$ uniquely extends to an isomorphism $\sqrt {G_1} \to \sqrt {G_2}$.

If  $G$ is finitely generated, then the group $\sqrt G$ has finite rank, i.e. there is an integer $r$ such that every finite subset of $\sqrt G$
is contained in an $r$-generated subgroup of $\sqrt G$.
\end{lemma}

There is an extensive array of literature devoted to the divisible groups and their generalizations (see a survey \cite{MR}). Divisible groups are a particular case of $A$-groups, where $A$ is an associative ring with 1. In these groups one can raise any element of the group to the power equal to an element of $A$, with natural axioms satisfied. Thus any group is a $\ZZ$-group while a divisible group is a $\QQ$-group. If this terminology is used then the isolated subgroups in divisible groups are called $\QQ$-subgroups. This terminology is used in \cite{KHU}, to which we give several references in the sections that follow.

\subsection{Commensurators of nilpotent groups}\label{acng}

Many groups are saturated with the subgroups of finite index. In a number of situations in Mathematics,  it is vital to consider not only the automorphisms of a given group $G$ but also the isomorphisms between the subgroups of finite index in $G$. Such isomorphisms are called a {\it virtual automorphisms} of $G$. We refer to the classical work \cite{Ma} of G. A. Margulis on discrete subgroups of Lie groups. For more references related to the concept we treat in this section, see, for example, \cite{BB} and \cite{DK}. 

Given another pair of subgroups of finite index $H'$ and $K'$ and a virtual isomorphism $\psi:H'\to K'$, the product of these partial mappings $\vp\psi$ is defined and its domain and range also have finite indices.
We say that $\vp$ and $\vp'$ are \textit{equivalent}, $\vp\sim\vp'$,  if there is a subgroup $L$ of finite index both in $H$ and $H'$ such that the restrictions of $\vp$ and $\vp'$ to $L$ are equal. Clearly, $\sim$ is an equivalence relation, actually a congruence, that is, if $\vp\sim\vp'$ and  $\psi\sim\psi'$ then $\vp\psi\sim\vp'\psi'$. As a result, the set of the congruence classes $\mathrm{Comm}(G)$ is a group called the  \textit{commensurator} of the group $G$.

For example, it is easy to check that $\mathrm{Comm}({\mathbb Z})$ is isomorphic
to the multiplicative group of rational numbers. The commensurator  of an arbitrary torsion free nilpotent group can be described as follows.

\begin{theorem} \label{pCom2} If $G$ is a finitely generated torsion free nilpotent group, then the  commensurator $\mathrm{Comm}(G)$ is isomorphic to the automorphism group $\Aut L$, where $L = \cL(\sqrt G)$ is the rational Lie algebra corresponding to the divisible hull $\sqrt G$ of $G$ in Theorem \ref{mal}.
\end{theorem}

\begin{proof} Let $\vp: H\to K$ be a virtual automorphism of $G$.
Since $H$ and $K$ have finite indices in $G$, it follow from the definition
of the divisible hull and Lemma \ref{hull} that $\sqrt G= \sqrt H=\sqrt K$. Therefore by Lemma \ref{hull}, the isomorphism $\vp$ extends to a unique automorphism $\bar\vp$
of the group $\sqrt G$.

Since equivalent virtual automorphisms $\vp$ and $\vp'$ coincide on a subgroup $H''$ of finite index in $G$, it follows (again by Lemma \ref{hull}) that the above-mentioned  extension mapping can be treated as a well defined function $f: \mathrm{Comm}(G)\to Aut(\sqrt G)$. Moreover, $f$ is a group homomorphism. For a nontrivial $\vp$, the extension $\bar\vp$ is nontrivial too, and so the mapping $f$ is injective.

Let $\alpha$ be an automorphism of $\sqrt G$. Then $\sqrt G$ is the divisible
hull of both $G$ and $\alpha(G)$. Hence for every $x\in \sqrt G$, there is
a positive integer $n$ such that $x^n\in K=G\cap\alpha(G)$. Since every subgroup
of a nilpotent group is subnormal and $G$ is polycyclic (see \cite{KM}, ch. 6),
$K$ must have finite index in $G$. This follows because otherwise the Hirsch rank of
$K$ (the number of infinite factors in a subnormal series with cyclic factors)
would be less than the Hirsch rank of $G$. In its turn, this would imply the existence
of $x\in G$ such that no nontrivial power of $x$ is in $K$, a contradiction. For the same reason, the subgroup $H=\alpha^{-1}(K)$ has
finite index in $G$. Hence $\alpha$ is an extension of the isomorphism
$H \to K$. Thus, the mapping $f$ is surjective hence an isomorphism.

Recall that by the BCH-formula
(\ref{BCH}), the group operation in $\sqrt G$ is derived from the algebra operations in $L$ . It remains to note that by Theorem \ref{mal}, one may assume that the group $\sqrt G$ and the Lie algebra $L$ have the same underlying set, and all algebra operations are derived from the group
operations (including raising to rational powers). It follows that
every automorphism of $L$ is an automorphism of $\sqrt G$, and vice versa.
This completes the proof.
\end{proof}

\begin{cy} For every finitely generated nilpotent group $G$,
the commensurator $\mathrm{Comm}(G)$ is a linear algebraic group over $\mathbb Q$.
\end{cy}

\begin{proof} For arbitrary finitely generated nilpotent group $G$, by Hirsch' theorem \cite{H}, there exists $n>0$, such that the (finitely generated) subgroup $G^n$ generated by all $n$-th powers of the elements is torsion free and has finite index in $G$.
So if $H\to K$ is a virtual automorphism of $G$ , then $H^n\to K^n$ is a virtual
automorphism of the subgroup $G^n$. This correspondence agrees with  the equivalence and the product, and so we get a homomorphism $\mathrm{Comm}(G)\to\mathrm{Comm}(G^n)$. It is injective since nonequivalent
$\vp$ and $\vp'$ cannot coincide on a subgroup of finite index. It is also surjective since every isomorphism $H\to K$, where $H$ and $K$ have finite
index in $G^n$ is equivalent to its restriction $H^n\to K^n$.

Thus, one may assume that the group $G$ is torsion free. By Lemma \ref{hull},
the group $\sqrt G$
has finite rank. It follows by the definition of the group operation in (\ref{BCH}) claim (4) that the mapping $\sqrt G\to L/(L,L)$, given by the formula $x\mapsto x+(L,L)$ is a group epimorphism. As a result, if $L$ is the rational Lie algebra $L=\cL(\sqrt G)$ then the factor algebra $L/(L,L)$ is finite dimensional. Now since $L$ is nilpotent, the entire algebra $L$ is finite dimensional.

By Theorem \ref{pCom2}, the group $\mathrm{Comm}(G)$ is isomorphic with the
group $Aut\; L$. For a  basis $(e_1,\dots,e_d)$ in a finite dimensional algebra $L$ over $\mathbb Q$,
the property that a nonsingular linear operator $\alpha: L\to L$ is an isomorphism is equivalent to the finite set of equalities $(\alpha(e_i),\alpha(e_j))=\alpha((e_i,e_j))$, which leads to a finite
system of algebraic equations imposed on the entries of the matrix of $\alpha$ with
respect to the basis $(e_1,\dots,e_d)$.
\end{proof}

{\bf Examples}.\begin{enumerate}
\item[1.] Let us explicitly produce the matrix form for the group of virtual isomorphisms  of the free 2-generator nilpotent group $G=G(x,y)$ of class 2 (the Heisenberg group). The group $\mathrm{Comm}(G)$ can be presented as the group of all nonsingular rational $3\times 3$-matrices $[a_{ij}]$ such that $a_{13}=a_{23}=0$, while  $a_{33}$ is equal to its complementary minor. This follows because this same group is the automorphism group of the free 2-generator nilpotent Lie algebra $L=L(x,y)$  of class 2. Any automorphism $\vp$ of $L$ is defined by the images of $x,y$,

\[\vp(x)=a_{11}x+a_{21}y+a_{31}(x,y), \,\vp(y)=a_{12}x+a_{22}y+a_{32}(x,y).\]
If we set $A=[a_{ij}]$ for $i,j \le 2$, then $\Delta=\det A\ne 0$. Also we set $u=[a_{31},a_{32}]$ and choose $x,y,z=xy$ as a basis of $L$. Then the matrix $[\vp]$ of $\vp$ with respect to this basis will have the form
\begin{equation}\label{eCom1}
[\vp]=\begin{pmatrix}
A&0\\u&\Delta
\end{pmatrix},
\end{equation}
where $\Delta = \det A$. Clearly, this is an algebraic group of dimension 6 over $\Q$.

\item[2.] In a similar manner one can determine the group $\mathrm{Comm}(G(m,c))$ of virtual automorphisms of  an $m$-generated free nilpotent group $G(m,c)$ of class $c$. This follows because, again, any map of the free generators of free  Lie algebra $L(m,c)$ of class $c$ to $L(m,c)$ extends to a Lie algebra homomorphism, which is an automorphism if and only it induces a nonsingular linear map on the space $L(m,c)/(L(m,c),L(m,c))$ over $\Q$. The automorphisms that are identical modulo $(L(m,c),L(m,c))$ form a nilpotent normal subgroup $N$ and $\Aut {L(m,c)}/N\cong \GL(m,\Q)$.

Next, it was proved in \cite{VML} and reproved in \cite{AM}, that the automorphism group of $UT(n,\Q)$ is generated by inner automorphisms, central automorphisms  (that is, identical modulo the center), diagonal automorphisms (conjugation by the diagonal matrices) together with the flip, which is the refection with respect to the anti-diagonal. It follows that these automorphisms generate  the group isomorphic to $\mathrm{Comm}(UT(n,\Z))$, because $UT(n,\Q)$ is the divisible
hull of the group $UT(n,\Z)$.

\item[3.]One more series of examples is provided by ``filiform'' nilpotent groups $G$ of nilpotent class $c>2$, which are semidirect product of an infinite cyclic group $\langle a\rangle$ and a free abelian group $\ZZ^c$. The matrix of the action of $a$ is  the Jordan cell with 1 on the diagonal. The divisible hull $\sqrt{G}$ is the semidirect product of $\QQ$ as the hull for  $\langle a\rangle$ and $\QQ^c$ as the hull for $\ZZ^c$. If $b$ is the generator for the $\ZZ^c$ with respect to the described action of $a$, then $\sqrt{G}$ is generated, as a divisible group, by $a$ and $b$. Mapping $a$ to any element outside $\QQ^c$ and $b$ to any element $y$ outside $(a-1)\QQ^c$, we obtain an automorphism of $\sqrt{G}$ and $\cL(\sqrt{G})$. Conversely, under every automorphism, $a$ should stay outside of an abelian group $\QQ^c,\:(c>2)$. This follows because the centralizer of $a$ has rank 2. At the same time,  $b$ should stay outside $(a-1)\QQ^c$. This completely describes $\Comm(G)$ as an algebraic group of dimension $c+1+c=2c+1$.
\end{enumerate}

\subsection{Determinants in  commensurators.}

One more fact is worth mentioning. Since the automorphism group of an algebra is a subgroup of a matrix group, every automorphism has determinant. Let $H$, $K$be subgroups of finite index in a finitely generated nilpotent group $G$ and $\vp$ is an isomorphism $\vp:H\to K$/ The $vp$ defines an automorphism of $L = \cL(\sqrt{G}$. of a , defined  with the help of the subgroups , we define its determinant $\det\vp$ as the determinant of the respective automorphism of the divisible hull $\sqrt{G}$, viewed as a Lie algebra.

\begin{theorem}\label{tComm} The absolute value of the determinant of a virtual automorphism $\vp:H\to K$ of a finitely generated torsion free nilpotent group $G$ is the ratio of the index of $K$ to the index of  $H$.
\end{theorem}

As a quick example, the map $2\Z\to 3\Z$ ($2\mapsto -3$) in the group of integers $\Z$ is given by the matrix $\left[-\frac{3}{2}\right]$, whose determinant equals $-\frac{3}{2}$.

\begin{proof} We start with a finitely generated additive torsion-free abelian group $G$. Choose a basis $( e_1,\dots,e_n)$ of $G$. The same set will be the basis over $\Q$ in the divisible hull $\sqrt{G}$ of $G$. Any basis $( h_1,\dots,h_n)$ of a subgroup $H$ of finite index in $G$ is defined by a square matrix $A$ such that $(h_1,\dots,h_n)=(e_1,\dots,e_n)A$. It is well-known by the theorem on the subgroups of free abelian groups, that $[G:H]=|\det A|$. If we have another subgroup of finite index $K$ with a basis $( k_1,\dots,k_n)$, then $(k_1,\dots,k_n)=(e_1,\dots,e_n) B$, for a square matrix $B$, and $[G:K]=|\det B|$.

Any isomorphism $H\to K$ extends to an automorphism of $\sqrt{G}$ such that the map of the bases $(h_1,\dots,h_n)\mapsto (k_1,\dots, k_n)$ is given by the formula $(k_1, \dots, k_n)=(h_1,\dots,h_n)C$, where  $C=A^{-1}B$. The determinant of this automorphism of the group $\sqrt{G} $, or the respective abelian Lie algebra, equals $\det C=(\det A)^{-1}\det B$.

It follows that $[G:K]/[G:H]=|\det B/\det A|=|\det C|$. Thus the  ratio of the indexes of subgroups of finite index under  a virtual automorphism equals the absolute value of the determinant under an automorphism of the group $\sqrt G$.

In the case of non-abelian groups, we will proceed by induction. It should be reminded that in a torsion-free nilpotent group $G$,
if the elements $a$ and $x^n\,(n\ge 1)$ commute then also $a$ and $x$ commute \cite{KM}. As a result,
the factor-group $G/Z(G)$ of $G$ by the center $Z=Z(G)$ of $G$ is also torsion-free. By the same reason, for any subgroup $H$ of finite index in $G$ we have $H\cap Z= Z(H)$, where $Z(H)$ is the center of $H$.

One more formula is useful for our inductive argument.
\begin{eqnarray}\label{eCOM1}
[G:H]&=&[G:HZ][HZ:H]=[G:HZ][Z:H\cap Z]\nonumber\\
&=& [G/Z: HZ/Z][Z:Z(H)].
\end{eqnarray}

Now we are ready to proceed with the proof.

 An isomorphism $H\to K$  of the subgroups of finite index induces an isomorphism $Z(H)\to Z(K)$. Thus we have a well-defined isomorphism
 \[ H/Z\cap H =H/Z(H)\to K/Z(K)=K/Z\cap K
 \]
  as well as an induced isomorphism $HZ/Z\to KZ/Z$. By induction, $[G:KZ]/[G:HZ]$ is the absolute value of the determinant of the Lie algebra automorphism built using the group $G/Z$ while $[Z:Z(K)]/[Z:Z(H)]$ is the absolute value of the determinant of the automorphism of a Lie algebra built using the group $Z$ and its divisible hull.

Let $\sqrt{G}$ be the divisible hull of $G$. Recall that a nontrivial power of any  element of $\sqrt{G}$ is in  $G$.  Thus the divisible hull $\sqrt{Z}$ of the center $Z$ is the center of the divisible hull of $\sqrt{G}$, $\sqrt{Z}\cap G=Z$, and $\sqrt{G}/\sqrt{Z}$ is the divisible hull of $G/Z$.

If we view a Lie algebra $\mathcal{L}$ built on the elements of $\sqrt{G}$ by Theorem \ref{mal}, then $\sqrt{Z}$ is the center (invariant with respect to $\Aut L$). Thus the matrix of any automorphism with respect to a basis complementary to a basis of the center, is of the block shape, with the square matrices $A$ and $B$ and zeros under these blocks, but not above them. Here $A$ is the matrix of the restriction to the center, while $B$ is the matrix of the  induced automorphism of ${\mathcal{L}}/ \sqrt{Z}$.

For the isomorphism $K\to  H$,  we have 
\[
|\det A|=[Z:Z(K)]/[Z:Z(H)],
\]
and at the same time,
\[
|\det B |=[G/Z:KZ/Z]/[G/Z:HZ/Z].
\]

As a result, we derive from formula (\ref{eCOM1}) that the absolute value of the determinant of the automorphism equals
\begin{eqnarray*}
|\det C|&=&|(\det A)^{-1}( \det B)|\\&=& ([Z:Z(K)]/[Z:Z(H)])([G/Z:KZ/Z]/[G/Z:HZ/Z])\\ &=&[G:K]/[G:H],
\end{eqnarray*}
which is what we needed to prove.
\end{proof}

Conversely, if we have a finite-dimensional nilpotent Lie algebra $L$ over $\Q$ then using Baker--Campbell--Hausdorff's formula makes $L$ a nilpotent divisible group. This group is a divisible hull of a finitely generated torsion-free nilpotent group. So, we can say that the automorphism group $Aut L$ is at the same time a group of virtual automorphisms  of a finitely generated torsion-free nilpotent group

\textbf{An observation}.  In the first example of
Section \ref{acng}, the determinant of the automorphism of the 3-dimensional algebra was always the square of a rational number. Hence, the ratio of indexes of isomorphic subgroups of finite index in that group is always the square of a rational number. In particular, if $H$ is a subgroup of finite index in the free nilpotent group $G(x,y)$ and $H$ is isomorphic to the whole group $G(x,y)$ then the index $|G(x,y):H|$ is  a perfect square!

\subsection{Groups of polynomial mappings}\label{GU}

Let $A$ be a nilpotent algebra of class $c\ge 1$ over a field $\F$.
On $A$, we consider polynomial mappings $f:x\mapsto a_1x+a_2x^2+\ldots$, where
$a_1\ne 0$ and dots stand here for a linear combination of (nonassociative) monomials in $x$ of degrees $d\in\{3,\dots, c\}$.
 By Remark \ref{sssMO}, $f$ is bijective on $A$, $f^{-1}$ is a polynomial mapping, and so such mappings form a group $G=G(A)$ under the composition $(fg)(x)=f(g(x))$.

Since $A$ is nilpotent and nonzero, the coefficient $a_1$ is uniquely
defined by $f$. If $g:x\mapsto b_1x+b_2x^2+\ldots$, then we see that the product $fg$
in $G$ has the form $x\mapsto a_1b_1x+\dots$, and so the mapping $f\mapsto a_1$ is a homomorphism of $G$ onto the multiplicative group $\F^*$ of the field $\F$. The functions from the kernel $U=U(A)$ have the form $x\mapsto x+ax^2+\ldots$, and $G/U\cong\F^*$. Moreover
the functions $x\mapsto ax$ form a semidirect complement to $U$ in $G$. The group $U$
is torsion free if $\ch\;\F=0$, because for a function $f:x\mapsto x+ax^k+\dots$ with $k\ge 2$ and $a\ne 0$, we have $f^n: x\mapsto x+nax^2+\dots$. The same equality implies that $U$ is a $p$-group
if $\ch\;\F= p>0$.

Some auxiliary computations will be helpful below.

\begin{lemma}\label{0} (1) Let $U\ni f: x\mapsto \vp(x)$ and
$f':x\mapsto \vp(x)+\mu(x)$, where $\mu (x)$ does not contain monomials
of degree $\le k$. Then the product $f^{-1}f'$ (and $f'f^{-1}$)
has the form $x\mapsto x+\mu(x)+\dots$, where $\dots$ does not contain
monomials of degree $\le k$.

(2) For every $f,g\in U$, the commutator
$f^{-1}g^{-1}fg$ has the form $x\mapsto x + h(x)$, where $h$
does not contain monomials of degree less than $4$.

(3) If $f: x\mapsto x+u(x)$, where $u(x)$ contains monomials
of degrees $\ge k\ge 4$ only, and $g\in U$, then the
commutator $fgf^{-1}g^{-1}$ has the form $x\mapsto x+v(x)$,
where the polynomial $v(x)$ does not contain monomials of degree
$\le k$.
\end{lemma}
 \begin{proof}(1) The function $f^{-1}$ is given by $x\mapsto \lambda(x)$. Here $\lambda(x)=x+\dots$ is a polynomial such that
 $\lambda(\vp(x)) = x + \rho(x)$, where  $\rho$
 is a polynomial without monomials of degree $\le c$. It follows
 that $\lambda(\vp(x)+\mu(x)) = \lambda(\vp(x)) + \mu(x)+\dots$, as required.

 (2) For $f:x\mapsto x+ax^2+bx^2 x +c x x^2+\dots $ and $g:x\mapsto x+kx^2+lx^2 x+ m x x^2 \dots $
 we have \[fg(x)= x+kx^2+lx^2 x+ mx x^2+
 a(x+kx^2+lx^2 x+ mx x^2)^2+ \] \[
 b(x+kx^2+lx^2 x+ mx x^2)^2(x+kx^2+lx^2 x+ mx x^2)+ \] \[
 c(x+kx^2+lx^2 x+ mx x^2)(x+kx^2+lx^2 x+ mx x^2)^2+\dots =\]
 \[ x+(k+a)x^2+(l+ak+b)x^2 x+(m+ak+c)xx^2+\dots\]
 Up to the dots, we have the same expression for $gf$, because
 of its  symmetry with respect $(a,b,c)\leftrightarrow (k,l,m)$.
 Since ``$\dots$'' contains only monomials of degrees $\ge 4$,
 we have by statement (1), that the commutator $(gf)^{-1}(fg)$
 is of the form $x\mapsto x+h(x)$, where $h(x)$ does not contain
monomials of degree $\le 3$.

(3) Since $g\in U$, we obtain $fg: x\mapsto g(x) +u(g(x)) =g(x)+u(x) +\dots$,
where ``$\dots$'' contains no monomial of degree $\le k$.
 Also
$gf: x\mapsto g(x+u(x))= g(x)+u(x)+\dots$. Now by statement (1), we have $(gf)^{-1}(fg): x +v(x)$, as required.
 \end{proof}

 Let us denote by
$U_k$ the set of mappings $x\to x+u(x)$, where the polynomial $u(x)$ has no monomials of degree $\le k$ in $x$ . It is easy to see that every $U_k$ is an isolated subgroup of $U$. So $U_1=U$ and $U_c$ is the trivial subgroup.

\begin{prop} \label{11} The group $G$ is solvable, the subgroup $U$ is nilpotent of class $\le \max(1,c-2)$, and the series
\begin{equation}\label{c}
U=U_1\ge U_3\ge U_4\ge\dots \ge U_{c-1}\ge U_c=\{1\}
\end{equation}
is a descending central series in $G$.
\end{prop}
\begin{proof} It suffices to prove the latter statement.

By Lemma \ref{0} (2), we have $[U,U_1]\le U_3$, and by Lemma \ref{0} (3),
we obtain $[U,U_k]\le U_{k+1}$ for $k\ge 3$. Thus, the series
(\ref{c}) is indeed central.
\end{proof}

\begin{lemma} \label{di} The group $U=U(A)$ is divisible if $\ch\;\F=0$.
\end{lemma}
\begin{proof}
By Propostion \ref{11}, $U$ has a descending series (\ref{c}). We first show that every central factor $U_k/U_{k+1}$ is divisible. Indeed, if $f: x\mapsto x+u(x)+\dots$
and $g: x\mapsto x+v(x)+\dots$, where $u(x)$ and $v(x)$ are homogeneous
polynomials of degree $k+1$ and dots stand for the  terms of higher
degree, then $fg: x\mapsto x+u(x)+v(x)+\dots$. In particular, for $n\ge 1$, we get
$f^n: x\mapsto x +nu(x)+\dots$, which implies the divisibility of the
factor $U_k/U_{k+1}$, since $n$ is invertible in $\F$.

To complete the proof, notice that if a central subgroup $Z$ of a group $G$ is divisible and $G/Z$ is divisible then also $G$ is divisible. So the divisibility of $U$ easily follows by induction on $c$ in (\ref{11}).
\end{proof}

\begin{rk} \label{act} For an arbitrary characteristic of the field $\F$, the same argument implies that all the factors $U_k/U_{k+1}$
are vector spaces over $\F$. The induced action by conjugation of every non-$k$-root of the identity ($k<c$) from the
factor group $G/U\cong F^*$, on $U_k/U_{k+1}$ does not have $1$ as an eigenvalue.
Indeed, let $g: x\mapsto ax$ with $a^k\ne 1$, and $f: x\mapsto x+u(x)+\dots$ represent nontrivial element of $U_k/U_{k+1}$, as above.
Then $g^{-1}fg: x\mapsto a^{-1}(ax +u(ax)+\dots)=x + a^k u(x)+\dots\ne x+u(x)+\dots$.
\end{rk}

\begin{cy}\label{Uder} If the field $\F$ is infinite, then $U$ is the derived
subgroup of $G$.
\end{cy}
\begin{proof} We have $[G,G]\le U$, because $G/U$ is an abelian group.
Since $\F$ is infinite, there is an element $a\in \F$, as in Remark \ref{act}. Then the action of $a-1$ is nonsingular on every central
factor of the nilpotent group $G$, and so $[a,U]=U$, whence
$[G,G]\ge U$.
\end{proof}

The properties of $G$ and $U$ obtained in this section will be
applied to the key examples in the next one.

\subsection{Filiform groups of mappings and filiform Lie algebras}\label{FGM}

Consider now the groups of polynomial mappings for power-associative algebras. For example, for $c\ge 3$,
let $A_c$ be the algebra of polynomials in $t$ over $\F$, $char\;\F=0$, with zero constant terms factorised over the ideal $(t^{c+1})$. It has
dimension $c$ and nilpotency class $c$ too. The group $G(A_c)$ acts
in the regular way on the orbit $O(t)$, and the functions $x\mapsto a_1x+a_2x^2+\dots + a_c x^c$ with different vectors of parameters
$(a_1,\dots,a_c)$ are different. In this case, the definition of the subgroup
$U(A_c)$ given in Section \ref{GU}, becomes a trimmed version for the definition of the group $G(\F)$ of formal power series under
substitutions introduced by D. L. Johnson \cite{J}. The Johnson's group $G(\F)$ is the projective limit of the groups $U(A_c)$, $c\to \infty$. If $\F= {\mathbf Z}/p{\mathbf Z}$, The Johnson's
group is called the {\it Nottingham} group, important in the
theory of pro-p-groups \cite{SSS}.

\begin{lemma}\label{2} (\cite{J}) (1) If $c>3$, $f: x\mapsto x+a x^{c-1}+bx^c$, and $g:x\mapsto
x+x^2$, where $x\in A_c$, we get $fgf^{-1}g^{-1}: x\mapsto x+(c-3)ax^c$.

(2) If $k\ge 3$, then we have $[U_k,U_k]\le U_{2k+1}$.

(3) If $char\;\F=0$ and $c\ge 3$, the series (\ref{c}) is the lower central series of the group $U(A_c)$, and so this group is nilpotent of class
$c-2$.

\end{lemma}

\begin{lemma} \label{F} For $c\ge 3$ and $U=U(A_c)$, the factor group $U/U_3$
is isomorphic to the additive group $\F \bigoplus\F$. If $3\le k\le c-1$ then the other terms
$U_k/U_{k+1}$ of the lower central series are isomorphic to the additive group of $\F$.
\end{lemma}

\begin{proof} Every function from $U$ is defined modulo $U_3$ by
two parameters $a,b \in \F$, namely, $f_{a,b}: x\mapsto x+ax^2+bx^3+\dots$. Let us denote the coset of $U_3$ containing $f_{a,b}$ by $\overline{(a,b)}$. 
If $g: x\mapsto x+kx^2+lx^3+\dots$, then by the formula from
the proof of Claim (2) in Lemma \ref{0} (now with $c=m=0$), we obtain
$fg: x\mapsto x+(k+a)x^2+(l+2ak+b)x^3+\dots$. So for the bijection
$\alpha: U/U_3\to \F\bigoplus\F$ given by the correspondence $\overline{(a,b)}\mapsto
(a, b-a^2)$, we get $\alpha(\overline{(a,b)})+\alpha (\overline{(k,l)})= (a+k, b+l-a^2-k^2)$.
The product $fg$ defines the pair $\overline{(a+k, b+l+2ak)}$, whence
\[
\alpha(\overline{(a+k, b+l+2ak}) = (a+k, b+l+2ak-(a+k)^2) = (a+k,b+l-a^2-k^2),
\] and so $\alpha$ is the desired group isomorphism.

For the functions $f:x\mapsto x+ax^{k+1}+\dots$ and $g:x\mapsto x+bx^{k+1}+\dots$
representing the elements of $U_k/U_{k+1}$ with $k\ge 3$, we have $fg: x\mapsto x+(a+b)x^{k+1}+\dots$, making the isomorphisms  $U_k/U_{k+1}\cong \F$ obvious.
\end{proof}

\begin{lemma} \label{N} Assume that $\ch\; \F=0$ and a normal subgroup $N$ of $U=U(A_c)$
contains a function $f:x\mapsto x+ax^k+\dots$ for some $k\ge 4$ and $a\ne 0$. Then the derived subgroup $[N,N]$ contains a function $x\mapsto x+a'x^{2k}+\dots$ with $a'\ne 0$.
\end{lemma}
\begin{proof} Let $f:x\mapsto x+ax^k+u(x)$, where all monomials in $u(x)$ have degree $\ge k+1$. There is nothing to prove if $2k\ge c+1$. Otherwise,
by Claim (1) of Lemma \ref{2} applied to the factor group $U(A_{k+1})$, the subgroup
$N$ contains a function $g: x\mapsto x+bx^{k+1}+v(x)$ for some $b\ne 0$, where all monomials in $v(x)$ have degree $\ge k+2$. We compute modulo $U_{2k}$ the products $fg$ and $gf$
\begin{eqnarray*}
fg: x&\mapsto& x+b x^{k+1}+ v(x)+ a( x+b x^{k+1}+v(x))^{k}+u(x+b x^{k+1}+v(x))\\ &=& x+ ax^k+bx^{k+1}+v(x) + kab x^{2k}+u(x)+\dots,
\end{eqnarray*}
 and
\begin{eqnarray*}
gf: x&\mapsto& x+ax^{k}+u(x)+b(x+ax^{k}+u(x))^{k+1} +v(x+ax^k+u(x))\nonumber\\&=&
x + ax^k +u(x) +b x^{k+1}+v(x)+ (k+1)bax^{2k}+\dots\nonumber
\end{eqnarray*}
So the difference $(gf)(x)-(fg)(x)$ is $ab x^{2k}+\dots$. By Claim (1) of Lemma \ref{0}, $[N,N]\ni (gf)^{-1}(fg): x\mapsto x +a' x^{2k}+\dots$ with $a'\ne 0$, as required.
\end{proof}

\begin{theorem} \label{gr} If $char\;\F=0$, Then the group $U(A_2)$ is abelian and for $c\ge 3$ the group $U(A_c)$ is nilpotent of class $c-2$. It is solvable of
length at least $k$ if $c\ge 2^k$; otherwise it is solvable
of length $<k$.
If the field $\F$ is infinite, then the solvability length of the group $G(A_c)$ is by one greater than the solvability length of $U(A_c)$.
\end{theorem}
\begin{proof} The statement on the nilpotency class follows from
Lemma \ref{0} (3).

Let $U^{(0)}=U$ and by induction, $U^{(s)}=[U^{(s-1)}, U^{(s-1)}]$
for $s\ge 1$. Then $U^{(1)}\le U_3$ by  Claim (2) of Lemma \ref{0} while Claim (2) of Lemma
\ref{2}  provides us, by induction, with inclusions  $U^{(k-1)}\le U_{2^k-1}$ for $k\ge 3$. So this  subgroup is trivial if
$2^k-1\ge c$, and the solvability length of $U$ does not exceed
$k-1$ if $c<2^k$.

Lemma \ref{N} and the obvious induction show that the $k$-th derived
subgroup $U^{(k)}$ contains an element from $U_{2^k-1}\backslash U_{2k}$ if $2^k\le c$. Thus, in this case, the solvability length of
$U$ is at least $k$.

The second statement is contained in Corollary \ref{Uder}.
\end{proof}

If $\F =\mathbb Q$, then by Lemmas \ref{di} and \ref{11}, the torsion free group $U=U(A_c)$ is a divisible nilpotent group. Hence, according to Malcev's correspondence, we have the rational nilpotent Lie algebra $L= \cL(U)$ with algebra operations defined on the same set. It follows from Lemma \ref{F} that $\sqrt U_k=U_k$ for each $k$. We refer to \cite[Theorem 10.13]{KHU} for the claims that follow. 

So every
$U_k$ becomes a subalgebra $L_k=\cL(U_k)$ of $L$ and an ideal of $L$.
Since the factors $U_k/U_{k+1}$ are central, the addition in $L_k/L_{k+1}$ coincides with the group operation in $U_k/U_{k+1}$,
and so it is a $\mathbb Q$-vector space of dimension 1 or 2 by Lemma
\ref{F}. Therefore $\dim L=\sum_k \dim L_k/L_{k+1} = 2+(c-3)\times 1=c-1$, and the nilpotency class of $L$ is $c-2$ by Theorems \ref{gr} and
\ref{mal}.

It is well known, that the factor-algebra  $L/[L,L]$ of a nilpotent Lie algebra $L$ of dimension $d\ge 2$ has dimension at least 2  and so the nilpotency class of $L$ is at most $d-1$. If it is exactly $d-1$,
then all other nonzero factors of the lower central series of $L$ have
 to be one-dimensional. Such algebras are called {\it filiform}. Filiform Lie algebras are an important block in the classification theory of finite-dimensional nilpotent algebras. They are also important in the theory of nil-manifolds in Differential Geometry. For a book and a survey on this subject see \cite{GK} and \cite{Rem}.

 Thus, we obtain a series of filiform Lie algebras $\cL(U(A_c))$, $c=3,4,\dots$.

 Since the solvability of length $k$ in Lie algebra is defined by a commutator of length $2^k$,  a nilpotent Lie algebra of class $n\ge 1$ is solvable with derived length $\le 1+\log_2 n $, and every nilpotent Lie algebra of dimension $\le 2^k$ has solvability length at most $k$.

 The algebra $\cL(U)$ has the same solvability length as the group $U$ (see \cite[Theorem 10.13, Claim (e)]{KHU}). So
 by Theorem \ref{gr}, the derived length of the filiform algebra
 $L(k)= U(A_{2^k})$ is $k$, the maximal derived length among the
 nilpotent algebras of the same dimension $2^k-1$. Thus, we have proved the first claim of the following.

 \begin{prop}\label{pBurde} The filiform nilpotent Lie rational Lie algebras $L(k)$
 have dimensions $2^k-1$ and derived length $k$ for every $k\ge 2$. There exist algebras of dimension  $2^k$ and solvability class $k+1$.
 \end{prop}
 
 \begin{proof} Let us  prove the second claim.  Remark \ref{act} and Lemma \ref{hull} we have a well-defined action of the group $\F^*$ on each factor $U_s/U_{s+1}$. If we identify this factor with the additive group of $\F$ (see Lemma \ref{F}), the action of each $a\in F^*$ is the multiplication by $a^s$, that is the scalar multiplication via the weight function $\nu_s: a\mapsto a^s$ defined on $F^*$.

It is mentioned just before Proposition \ref{11} that the subgroups $U_s$ are isolated.  It  follows by \cite[Theorem 10.13, Claims (a,f)]{KHU}), they are subalgebras in $\cL(U)$ invariant under the action of $F^*$. (Alternatively one could use Remark \ref{sub} and Theorem \ref{mal}).  The isomorphism of the groups $U_s/U_{s+1}$ and $1$-dimensional factor space $U_s/U_{s+1}$, established in  Lemma \ref{F} shows that the action of $F^*$in this subspace has the same weight $\nu_s$.

Thus, $\cL(U)$ has a nonzero weight space for each of the wieghts $\nu_1,\dots,\nu_{c-1}$. Since their number equals the dimension of $\cL(U)$, all the respective weight space are $1$-dimensional and $\cL(U)$ is their direct sum.

Since the product of two vectors with weights $\nu_s$ and $\nu_t$ in $\cL(U)$ has weight $\nu_{s+t}$, the weight subspaces determine a $\ZZ$ grading of $\cL(U)$, which provides us with an \textit{adapted basis} of $\cL(U)$, that is a basis with the condition $e_s e_t = \lambda_{s,t} e_{s+t}$  for some $\lambda_{s,t}\in \F$.

It is easy to see that the linear map $D$ defined on the adapted basis by $D:e_s \mapsto  s e_s$ is the derivation of $\cL(G)$. Since $D$ is nonsingular, $D(\cL(G))=\cL(G)$. If we consider the semidirect product $M=\langle D\rangle_\F\oplus\cL(U)$ then $D\cL(U)=\cL(U)$ in $M$ and so the derived subalgebra of $M$ equals $\cL(U)$. Then the derived length of $M$ is by 1 greater than the derived length of $\cL(U)$ so that  $M$ is  the desired algebra of dimension $2^k$ and derived length $k+1$.
 \end{proof}

Note that the filiform nilpotent Lie algebras $\mathfrak{f}_{\frac{9}{10},n}$ with the same parameters, that is, of dimension  $2^k-1$  and maximal solvability length $k$ for the algebras of this dimension, have been obtained earlier (see D. Burde's papers \cite{BDV}, \cite[Proposition 3.1]{Bur} and also \cite{BST}), in the context of Differential Geometry. Our second example in Proposition \ref{pBurde} matches another Burde's example \cite[Proposition 5.2]{Bur}. The results in the quoted paper are mostly obtained by direct computations, using the structure constants. Our algebras $L(k)$ admit a simple definition via Malcev's correspondence.

\subsection{Partially ordered algebras and (quasi)groups}\label{ssPOA}

The concept of a partially ordered group is well known. Namely, a multiplicative group $G$ is partially ordered if it is equipped with a partial order $\le$, and for arbitrary $a,b,c\in G$ the
inequality $a\le b$ implies $ac\le bc$ and
$ca\le cb$ (the multiplication is monotone).
The order is linear if either $a\le b$ or $b\le a$
for arbitrary $a,b\in G$.

Extending this definition to quasigroups, one
requests in addition that $ac\le bc$ implies
$a\le b$ and $ca\le cb$ implies $a\le b$.

The following definition of the order on algebras was
introduced by V. M. Kopytov \cite{KVM}. He suggested that a Lie algebra $L$ over a linearly ordered field $\F$ should be called \textit{partially ordered} if $L$ is a partially ordered vector space $(L,\le)$ such that for any $a,b\in L$, $a>0$, one has $ab< a$. Kopytov studied
linearly ordered Lie algebras. To justify the
definition, he showed that the group $\cG(L)$ defined
on a rational partially ordered nilpotent Lie algebra $(L, \le)$ by the
BCH-operation (\ref{BCH}) is linearly ordered with the same order $\le$.

Generalizing Kopytov's definition to arbitrary algebras, one has to require that $a>0$ implies
that $ab< a$ and $ba<a$. The set of
polynomial operations can also be much extended:

\begin{theorem} Let $(A, \le)$ be an arbitrary nilpotent, linearly ordered
algebra over a linearly ordered field $\F$, and $Q(A)$ be a quasigroup with a polynomial operation $\circ$ of the form (\ref{bino}). Then $Q(A)$ is a linearly ordered
quasigroup with the same order $\le$.
\end{theorem}

\begin{proof} For a non-negative $x$ and a positive $y$ in $A$, one writes $x\ll y$ (``much less'') if
$\alpha x < y$ for every $\alpha\in \F$.
It follows from this definition that if each of $x_1,\dots, x_k$ is much less than $y$, the
arbitrary linear combination $\sum_{i=1}^k \lambda_ix_i$ is less than $y$.

Let $a>b$ in $A$, i.e. $a = b+u$ for some $u>0$. Then 
\[
a\circ c = (b+u)\circ c = b\circ c +u  + h(b,c, u),\]
where every monomial $w$ of the polynomial $h$
involves $u$ and at least one more factor.
By induction and the definition of an order
on algebra, we have that $|w|\ll u$, where $|w|=
\max(w,-w)$. Hence we have $|h(b,c, u)|<u$ for
the linear combination monomials, and so
$a\circ c> b\circ c$, as desired. Similarly
one obtains $c\circ a>c\circ b$.

Conversely, with the preceding notation, the inequality $a\circ c> b\circ c$
implies $u>h(b,c,u)$. If $u>0$, then $a>b$, as required. Otherwise $u<0$ since the order on $A$
is linear, and $-u>0$. But in this case the
argument of the previous paragraph provides us
with inequality $-u >h(b,c,u)$, i.e. $u<h(b,c,u)$,
a contradiction. Similarly, we have $a>b$ if $c\circ a> c\circ b$.
\end{proof}

We support Kopytov's definition of ordered algebra
by the following statement, which is the converse of the above-mentioned Kopytov's theorem.

\begin{theorem} Let $(G, \le)$ be an arbitrary nilpotent, linearly ordered divisible
group  and $\cL(G)$ be the conversion of $G$ into the rational Lie algebra,
according to Theorem \ref{mal}. Then $(\cL(G),\le)$ is a linearly ordered
algebra (notice that the order remains the same).
\end{theorem}

\begin{proof} We keep the notation $\circ$ for the product in $G$. For two  elements $x,y$ of $G$ the relation $x\ll y$ (``much less'') is defined if $|x^n|<y$ for every  integer $n$.
It follows from the definition that if  $x_1 \ll y$ and
$x_2\ll y$, then $x_1x_2\ll y$ and $x_1^q<y$ for any rational exponent $q$. Therefore $x\ll y$ implies $|qx|\ll y$ since for every rational $q$, the element $qx$ in $\cL(G)$ is just $x^q$ in $G$.

Since for every positive $y$ in a linearly
ordered group $G$, we have $|x^{-1}y^{-1}xy|\ll y $ for any $x\in G$ (\cite{Sm}), we have $|u|\ll y$ for arbitrary product $u$
of rational powers of commutators involving $y$.

To prove that $\le$ is a linear order on the vector space $\cL(G)$, one should show that for $a>b$ in $G$ we get
$a+c>b+c$, i.e. the element $(a+c)\circ(b+c)^{-1}$ is positive
in $G$.

At first, we get $a=b\circ y$ for a positive $y$. Then by \cite[Lemma 10.12, Claim (a)]{KHU}, $a=(b+y)\circ u$, where $u$
is a product as above, and so $|u|\ll y$.

Then  we obtain
$(b+y)\circ u=(b+y+u)\circ v$, where $|v|\ll|u|$. By BCH formula
(\ref{BCH}), the right-hand side is $b+y+u+v+w$, where
$w$ is a linear combination of Lie commutators involving $v$.
However if $|x_1|\ll y$ and $|x_2|\ll y$ in $G$, then
for the sum in $\cL(G)$, we have $|x_1+x_2|\ll y$ in $G$,
because $x_1+x_2 = x_1\circ x_2\circ u'$, where again $|u'|\ll|x_1|$.
Also, for a group commutator $[f,g]$, we have $[f,g]=(fg)\circ z$, where $z$ is a product of longer group commutators involving both $f$ and $g$ ( see \cite[Lemma 10.12, Claim (d)]{KHU}. Therefore the reverse induction on the commutator length shows that $|w|\ll|v|$.

So, to prove that $a=b+z$ for a positive in $G$ element $z$,
it suffices to show that for a positive $y$  and $|x|\ll y$ in $G$,  we have $y+x>0$ in $G$. This is true since again $y+x = y\circ x \circ v$, where $|v|\ll y$.

Now $a+c= (z+b+c)=z'\circ(z\circ(b+c))$, where again $|z'|\ll z$.
Therefore  $(a+c)\circ(b+c)^{-1}= z'\circ z>0$ , as required.

Arbitrary Lie product $ab$ is a product of rational powers
of group commutators involving both $a$ and $b$ (see \cite[Lemma 10.12, Claim (b)]{KHU}). Therefore, as above, $a\gg ab$ if $a$ is a positive element for the order $\le$. This completes the proof of the theorem.
\end{proof}

We see that the categories of nilpotent divisible, linearly ordered groups and nilpotent linearly ordered algebras are equivalent. Also, as shown in  \cite{KK}, an arbitrary linear order on a nilpotent group $G$ uniquely extends to a linear order on its divisible hull $\sqrt G$. So it is not surprising that the two theories are absolutely parallel.

\subsection{Refining partial orders on algebras}

Here we present some auxiliary stuff, needed for the next sections.

An ideal $I$ of a partially ordered algebra (po-algebra) $A$ is called
{\it convex} if for every $0\le b\le a$, where $a\in I$,
we have $b\in I$. A convex ideal $I$ defines the induced
partial order on the quotient $A/I$: by definition
$x+I\le y+I$ if there exist $a,b \in A$ and $a\in x+I, b\in y+I$,
such that $a\le b$. The relation $\le$ is well defined
 on the factor algebra $A/I$ over a convex ideal, and the canonical mapping $A\to A/I$ is a homomorphism of po-algebras
 with kernel $I$.

 We quote the following result from \cite{KS} (also see \cite{KVM}).

\begin{lemma}[Analogue of Levi's Theorem from Group Theory]\label{T2}
Let $A$ be an algebra over a partially ordered field $\F$ and $I$ an ideal in $A$ such that both $A/I$ and $I$ are partially ordered. If for any $a>0$, $a\in I$ and any $b\in A$ we have $ab,ba\le a$ then one can define a partial order on $A$ so that $I$ is a convex ideal and the partial orders on $A/I$ and $I$ are induced by the order on $A$.
\end{lemma}

It follows that an arbitrary nilpotent algebra over the field $\F$
is linearly orderable. Indeed, to refer to Lemma \ref{T2}, one can choose an arbitrary linear order on the annihilator $I$ of $A$, which is a vector space with zero multiplication, and a linear order on the algebra $A/I$, where
such an order exists by induction on the nilpotency class.

The converse statement holds for finite-dimensional algebras.
\begin{prop} \label{linnil}If a finite-dimensional algebra $A$ over a linearly ordered field $\F$ is linearly ordered, then it is nilpotent.
\end{prop}
\begin{proof} Following the paper \cite{KS}, we consider the set
$M_a$ for arbitrary positive $a\in A$. By definition, $M_a =\{x\in A\; |\; |x|\ll a\}$. Clearly, $M_a$ is a convex ideal,
which does not contain $a$.

One may assume that $ab\ne 0$ for some $a,b\in A$. Then $0\ne |ab|\ll a$, and so the ideal $M_a$ is not zero. Choosing such
an ideal with minimal nonzero dimension, we have $xy=0$ for every (positive) $x\in M_a$ and $y\in A$, because otherwise one would obtain
$0<\dim M_x<\dim M_a$ since $M_x\subset M_a$ and $x$ belongs to $M_a \backslash M_x$.

By induction on the dimension, we conclude that the linearly ordered quotient $A/M_a$ is nilpotent. Since the ideal $M_a$ annihilates $A$, the entire algebra $A$ is nilpotent.
\end{proof}

\begin{rk}\label{r10} A simple algebra $A$ with nontrivial multiplication admits only trivial ordering. Indeed, if $a>0$ in $A$, then $a\notin M_a$, so the ideal $M_a$ is $\{0\}$. For any $x\in A$, $xa,ax\in M_a$, so that $xa=ax=0$, and the annihilator of $A$ is a nontrivial ideal. As a result, $A$ has zero product, a
contradiction.
\end{rk}

\begin{theorem}\label{t1}
Let $A$ be a finite-dimensional algebra with nonzero product and nontrivial partial order $\le$. Then $A$ contains
a convex ideal $I$ such that $\{0\}\ne I\ne A$, and
the ideal $I$ ether annihilates $A$ or the restriction of $\le$
to $I$ is the trivial order.
\end{theorem}
\begin{proof}
Assume first that all positive elements $a\in A$
are annihilating, i.e. $ax=xa =0$ for every $x\in A$.
The annihilator $I$ of $A$ is a convex ideal since the positive elements all lie in $I$. It is nonzero since the order is nontrivial, and it is proper since the multiplication is not trivial.

Now suppose there is a non - annihilating positive element $a$. Being non - annihilating means that there is $x\in A$ such that one of $ax, xa\ne 0$. Since both $ax\ll a$ and $xa\ll a$,  $M_a$ is a nonzero ideal. Now $I=M_a$ is a convex ideal,
which does not contain $a$,
and so $I\ne A$. We may assume that the element $a$ is chosen
so that the ideal $I$ with these properties has minimal dimension.

Thus, $I$ is the desired ideal unless the restriction of the order $\le$ to $I$ is nontrivial and $I$ does not annihilate $A$. In this case, we have a positive element in $I$, and every
positive element $x\in I$ annihilates $A$, since otherwise
we would have $0\ne |xy|\ll x$ for some $y$ and $\{0\}\ne I' < M_a$ for the convex ideal $I'=M_x$ contrary to the choice of $a$.

Then we denote by $I''$ the linear envelope of all positive elements from $I$. We see that $I''$ is an annihilating ideal of $A$ and $\{0\}\ne I'' \le I<A$. It follows from the definition that $I''$ is convex in $I$, and so is in $A$, because, in turn,
$I$ is convex in $A$.

The theorem is proved.
\end{proof}

The following Kopytov's result was presented in \cite[Theorem 4.10]{KVM}.

\begin{lemma}\label{r11} Any partial order on a nilpotent Lie algebra $A$  over a linearly ordered field $\F$ can be extended to a linear order on $A$.
\end{lemma}

\subsection{Ranks of maximal partial orders.}

Here we focus on finitely-dimensional algebras and study their partial orders since by Proposition \ref{linnil}, only nilpotent algebras admit linear orders.

A partial order $\le'$ is called a {\it refinement} of a partial order $\le$ on a set if $x\le y$ implies $x\le' y$ for arbitrary
elements $x,y$. A partial order is called {\it maximal} if
it does not admit a proper refinement. Zorn's lemma implies that
for an arbitrary partial order on an algebra, there exists a maximal
refinement. Therefore every algebra admits maximal partial orders.

A partial  order $\le$ on an algebra $A$ is called {\it lexicographic} with respect to a finite series of convex ideals
\begin{equation}\label{ser}
0=I_0<I_1<\dots<I_m=A,
\end{equation}
if for every pair of elements $x\ne y$ and the maximal $k$
such that $x+I_k\ne y+I_k$, we have $x\le y$ or $x$ and $y$ are incomparable if and only if the same relation is true
for the cosets $x+I$ and $y+I$ with respect to the induced order
on the quotient $A/I_k$.

A quotient $I_{j+1}/I_{j}$ is called a {\it chief } factor for a series of ideal (\ref{ser}) (not necessarily convex) if it is a minimal (nonzero) ideal in $A/I_j$. The quotient $I_{j+1}/I_{j}$ is called a {\it central} factor here if
it annihilates the factor algebra $A/I_j$.

\begin{theorem} \label{factors} For arbitrary maximal partial order $\le$ on a finite-dimensional algebra $A$ over an linearly ordered field $\F$, there is a finite series of convex ideals (\ref{ser}) such that
every factor $I_k/I_{k-1}$ is either central and the induced order on it is linear, or a noncentral chief factor, and the induced
order on it is trivial.

Furthermore, the order $\le$ is lexicographic with respect to (\ref{ser}).
\end{theorem}

\begin{proof} We will proceed by induction on $d=\dim A$ with
obvious base $d=0$.

For $d>0$, we set $I_0=\{0\}$, assume first that the order $\le$ is trivial,
and choose a minimal ideal $I_1=I$. Then it is convex, the factor
$I/I_0$ is a chief factor in $A/I_0$, and $I_1/I_0$ is not
a central factor for the following reason. A central factor has trivial multiplication, and so it admits a linear order $\le_I$ by \cite{KVM}, 4.10. By Lemma \ref{T2}, $\le_I$ is a restriction of a nontrivial partial order on $A$, contrary
the assumption that the trivial order is maximal. Thus the factor
$I_1/I_0$ satisfies the requirements of the theorem.

Having trivial order, the ideal $I=I_1$ is convex, and the induced order $\le_{A/I}$ is maximal, because any proper refinement $\le'_{A/I}$ of it defines the proper refinement $\le'$ of $\le$ by the rule $x<'y$ if $x+I<'_{A/I} y+I$.
(This a refinement indeed, since $x<y$ implies $x-y\notin I$,
and so $x+I<y+I$.) Since $\dim A/I<d$, we get a required series
of convex ideals
\begin{equation}\label{AI}
I_1/I <I_2/I<\dots<A/I
\end{equation}
 in the quotient $A/I$, where the induced order on $A/I$ is lexicographic with respect to this series. Thus,
$I_0<I_1<I_2<\dots$ is a required series of convex ideals in $A$.

Let now the order $\le$ be nontrivial. If the multiplication
in $A$ is trivial, then $\le$ is a linear order by Lemma \ref{r11}, and the statement of the theorem holds for the series of ideals $\{0\}<A$.

In the remaining cases, we consider the convex ideal $I$ provided by Theorem \ref{t1}. If the order on $I$ is trivial, then we may
assume that $I$ is a minimal ideal since,  in this case, every ideal contained in $I$ is convex. The chief factor $I/\{0\}$ is not central. Indeed, otherwise there is a linear order $\le_I$ on $I$. Together with the induced order $\le_{A/I}$ it defines an order $\le'$ on $A$ by Lemma \ref{T2}, and $\le'$ is a proper refinement of $\le$, since it is nontrivial on $I$.

Therefore by induction, the series (\ref{AI}) allows us to obtain the required series (\ref{ser}), as above.

If the order on $I$ is not trivial, and so $I$ annihilates $A$,
the factors $I/\{0\}$ is central in $A$. Then again one could
extend the order on $I$ to a linear order and refer to Lemma \ref{T2} to obtain a partial order refining $\le$. Since the partial order is maximal, this refining is not proper, and so the restriction of $\le$ to $I$ is linear. The induced order on $A/I$ is maximal for the same reason. Then the induction on $d$ completes the proof as above.
 \end{proof}

 Given a maximal partial order $\le$ on a finite-dimensional algebra over a linearly ordered field, we call the sum of the dimensions
 of all central factors, provided by Theorem \ref{factors},
 the {\it rank} of the order $\le$.

 \begin{cy} \label{r} The rank of a maximal partial order on a
 finite-dimensional algebra $A$ over a linearly ordered field
 does not depend on the choice of the series (\ref{ser}) in  Theorem \ref{factors}. All maximal orders on $A$ have equal ranks.
 \end{cy}
 \begin{proof} The series (\ref{ser}) has a refinement with the same noncentral factors, where
 every factor is a chief factor  (The ideal of the refinement are not necessarily convex.) By Jordan - H\" older Theorem, the sets
 of factors of two such series (with multiplicities) are
 equal. Furthermore, we can use the operator version of this theorem, where the operators are operators of left and right multiplications by the elements of $A$. Thus, the sets of noncentral chief factors are the same. Since the rank
 of $\le$ given by (\ref{ser}) is the difference of $\dim A$
 and the sum of the dimensions of all noncentral factors in (\ref{ser}), both statements of the corollary are proved.
 \end{proof}
 
 
 \begin{df}\label{porank}
 The common value of all ranks of maximal orders on a finite-dimensional algebra $A$ is called the \textit{partial order rank of $A$} or \textit{p.o. rank of $A$}. 
 \end{df}

Now the following is true.

 \begin{cy} If a finite-dimensional algebra $A$ over a linearly ordered field admits a linear order, then arbitrary partial
 order on $A$ extends to a linear order. Any partial order on a locally nilpotent algebra extends to a linear order.
 \end{cy}
 
 The last claim generalizes Lemma \ref{r11}, proven for nilpotent Lie algebras.
 
 \begin{proof}
 The first claim and the second in the case of finite-dimensional algebras follow from the second claim of Corollary \ref{r}. Now if $A$ is locally nilpotent, then it has a local system of finitely generated nilpotent, hence finite-dimensional subalgebras, that is, such a system of subalgebras whose union is the whole algebra while any two terms of the system are contained in a third one.
 
 Now the possibility of extending any partial order to a linear order is a local property: if an algebra $A$ has a local system of subalgebras with this property, then $A$ has this property.  The proof that this is indeed a local property had been proven in the case of groups (see, for example, \cite[Addendum, $\S 2$, Theorem 2]{KM}). Since the proof only uses the general properties of the order relation and Malcev's Local Theorem, it transfers without changes to the case of algebras. 
 \end{proof}

 One more characterization of order rank is given by
 \begin{theorem} \label{rd} The rank of a maximal partial order $\le$ in a
 finite-dimensional algebra over a linearly ordered field is
 equal to the dimension of every subspace $V$ maximal with respect the property that the restriction of $\le$ to $V$ is a
 linear order.
 \end{theorem}

 \begin{proof} Let $V$ be as in the formulation of the lemma.
 On one hand, the subspace $V$ cannot contain a vector $a\in I_j\backslash I_{j-1}$
 if a quotient $I_j/I_{j-1}$ in the series (\ref{ser}) given in Theorem \ref{factors}, has trivial induced order. Indeed, the cosets $a+I_{j-1}$ and $I_{j-1}$ are incomparable, and so
 $a$ and $0$ are, because the order $\le$ is lexicographical with
 respect to the series (\ref{ser}). This contradicts the assumption  that the
 order $\le$ is linear on $V$. Then an induction on $k$
 shows that $\dim(V\cap I_k)$ is at most the sum of the
 dimensions of the central factors of $A$ in the subseries $\{0\}< I_1<\dots<I_k$. For $k=m$, we have that $\dim V$ does not
 exceed the rank of $\le$.

 Now we want to prove that $V$ must contain $n_j$ vectors of $I_j$ linearly independent modulo $I_{j-1}$, where $n_j=\dim(I_j/I_{j-1})$ and the quotient $I_j/I_{j-1}$ is a central factor of $A$.
 Arguing by contradiction, we choose a vector $u\in I_j\backslash V$ and consider the larger subspace $U=V+\langle u\rangle$, where
 $\langle u\rangle$ is the span of $u$.

 If $x\in U\backslash I_j$, then $0\ne x=y\; (\!\!\!\!\mod \;I_j)$,
for some $y\in V$. Since $y$ is comparable with $0$ in $A$, so is $x$,
because  the order $\le$ is lexicographic. If $x\in I_j\backslash I_{j-1}$ then $x$ is also comparable with $0$, for the same reason,
because the order on the factor $I_j/I_{j-1}$ is linear. If
$x\in U\cap I_{j-1}$, then we have $x\in V$, and so $x$ is again comparable with $0$, since the order $\le$ is linear on $V$.

Thus, the order $\le$ is linear on the larger subspace $U$, a
contradiction.  Therefore $\dim V$ is at least the sum of
 the dimensions of the central factors in (\ref{ser}), and the theorem is proved.
 \end{proof}

 \subsection{Maximal partial orders in solvable Lie algebras}

It turns out that the rank of maximal partial orders on finite-dimensional solvable Lie
algebras $A$  is equal to
the  (common) dimension of the Cartan subalgebras of $A$. Recall that in a finite-dimensional Lie algebra $A$ there is a nilpotent
subalgebra $H$ equal to the normalizer of it in $A$. It is called a \textit{Cartan subalgebra} of  $A$. In characteristic $0$, all Cartan subalgebras of $A$ have the same dimension called
the \textit{rank} of $A$ (\cite[Chapter VII, \S 3]{Bou}).

\begin{theorem} If $A$ is a finite-dimensional solvable Lie
algebra over a linearly ordered field, then the partial order rank on $A$ is equal to
the rank of $A$.
\end{theorem}
\begin{proof} In characteristic $0$, a finite dimensional Lie algebra $A$ is solvable if and only if its derived subalgebra $(A,A)$ is nilpotent (\cite[Chapter I, \S 3]{Bou}.
So the adjoint action of $(A,A)$ on the factors of a chief series
going through $(A,A)$ is trivial. Hence by Jordan - H\" older Theorem for algebras with operators, the derived subalgebra
acts trivially on an arbitrary chief factor $M$ of $A$, i.e. the
adjoint action of the algebra $A$ on $M$ is,
in fact, the action of the abelianizer $A/(A,A)$. Every chief factor of a solvable Lie algebra is an abelian algebra. It follows
that the kernel of the action of a particular adjoint operator $\ad\; a$ on the simple $A/(A,A)$-module $M$ is invariant with respect
to the adjoint action of $A$, and so the induced action of
$\ad\; a$ on $M$ is either nonsingular or trivial.


By Corollary \ref{r}, we may now consider arbitrary
maximal partial order $\le$ on $A$ and arbitrary series (\ref{ser}) given in Theorem \ref{factors}. We also consider
a refinement ${\mathcal{J}}:J_0<J_1<\dots $ of (\ref{ser}) with the same noncentral factors and one-dimensional central factors in the
chief series $\mathcal{J}$. So by Theorem \ref{factors}, the rank $r$ of the order $\le$ is just the
number of central factors in $\mathcal{J}$.

Since the adjoint action $\ad\,a$ of every element $a\in A$ is trivial in every central factor of $\mathcal{J}$,
the dimension of the null-component of the adjoint linear operator $\ad\, a: x\mapsto ax$ on $A$ is at least  $r$. (Recall that
the null-component consists of all vectors annihilated by some
power of the operator $\ad\, a$.)

For noncentral factors $J_k/J_{k+1}$ of $\mathcal{J}$, the action of $A$ is nontrivial, and so the kernel $K_j$ is a proper subspace
of it. The union of all (although of finitely many) $K_j$-s is a proper subset of $A$ too, because the ground field is infinite. Therefore
there is an element $a\in A$ acting nontrivially, and therefore
 in a nonsingular way, on every noncentral factor of $\mathcal{J}$.
 Hence the multiplicity of the root $0$ in the characteristic polynomial of $\ad\, a$ is exactly $r$. Hence the dimension of the
 null-component of the linear operator $\ad\, a$ on $A$ is minimal.

 An element $a$ with minimal dimension of the null-component
 for the action of $\ad\, a$ on $A$, is called regular. The
 null-component of $\ad\;a$ for a regular element $a$ is a Cartan subalgebra $H$ of a finite-dimensional Lie algebra $A$ (\cite{Bou}, VII.3). Thus, $\dim H =r$, as required.
 \end{proof}

 \begin{rk} Since $\dim H$ is the sum of dimensions of central factors
 in (\ref{ser}) and the regular element $a\in H$ provides a nonsingular action of $\ad\, a$ on  noncentral factors of the series (\ref{ser}), we see that the factors $(H\cap I_j)/(H\cap J_{j+1})$ are isomorphic with the  central factors $I_j/I_{j+1}$, as ordered spaces. Therefore the restriction
 of $\le$ to $H$ is a linear order,  and
 a linear order on $H$ uniquely defines the order on the
 central factors of (\ref{ser}), and so on the entire Lie algebra $A$ by Theorem \ref{factors} .(Also by Theorem \ref{rd}, the Cartan subalgebra $H$ is a maximal subspace of $A$, where the restriction of the maximal partial order $\le$ is linear.)

 However not every linear order on a Cartan subalgebra extends
 to a maximal partial order of the whole finite-dimensional
 solvable Lie algebra $A$ over a linearly ordered field. For an example, we take the following $4$-dimensional Lie algebra $A$.
This is the semidirect product of a $3$-dimensional nilpotent algebra $L$
 with basis $\{a, b, ab\}$ by a 1-dimensional subalgebra $\langle c\rangle$ such the the inner derivation $\ad\, c$ maps $a$ to $a$, $b$ to $- b$, and $ab$ to $0$. The subspace $H=\langle c, ab\rangle$ is a  $2$-dimensional Cartan subalgebra of $A$.

 It is easy to see that there are only two nonzero central
 factors  $A/L$  and\\  $\langle ab\rangle/\{0\}$ in $A$, and the latter
 factor must precede the former one in a series (\ref{ser}).
 So by Theorem \ref{factors}, any maximal (lexicographic)
 partial order on $A$ induces a lexicographic order on $H$
 with respect to the series $\{0\}\subset \langle ab\rangle\subset H$.  However $H$ admits
 many other linear orders for example any lexicographical order with respect to the series $\{0\}\subset \langle c\rangle\subset H$.
 \end{rk}


\begin{thebibliography}{99}
\label{bibl}
\bibitem[Ba]{Ba} Bahturin Yu. A., Identical relations in Lie algebras.
Second edition. De Gruyter Exp. Math. \textbf{68}(2021), xxv+514 pp.
\bibitem[BB]{BB} Bartholdi L., Bogopolski O., \textit{On abstract commensurators
of groups}. J. Group Theory \textbf{13}(2010), 903-922.
\bibitem[BG]{BG} Baumslag G., \textit{Some aspects of groups with unique roots}.
Acta Math. \textbf{104}(1960), 217 - 303
\bibitem[BST]{BST} Benoist, Yves, \textit{Une vari\'et\'e non affine}. J. Diff. Geom.
\textbf{41}(1995), 21 - 52
\bibitem[Bou]{Bou} N. Bourbaki, Elements de mathematique, Groupes et Algebres de Lie, Herman,
Chapter II,  1971, 1972, Chapter VII, 1975.
\bibitem[Bruck]{BRU} Bruck, R. H., A survey of binary systems, Ergebnisse der Mathematik und ihrer Grenzgebiete, Heft 20, Berlin, Springer, 1958,
185 pp.
\bibitem[BDV]{BDV}Burde, Dietrich; Dekimpe, Karel; Vercammen, Kim, \textit{ Novikov algebras and Novikov structures on Lie algebras}. Linear Algebra Appl. \textbf{429}(2008), 31 - 41
\bibitem[Bur]{Bur}Burde, Dietrich, \textit{Derived length and nildecomposable Lie algebras}.
Proc. 13th Int. Symp. Math. Appl. 2012,
Sci. Bul."Politehnica" Univ. Timisoara \textbf{58}(2013), 15-24 .
\bibitem[CGV]{CGV} Chanhan D., Gupta I., Verma R., \textit{Quasigroups and their applications in cryptography}. Cryptologia \textbf{45}(2021), 227-265.
\bibitem[CdGV]{CdGV} Cicalo S., de Graaf W., Vaughan-Lee M., \textit{An effective version
of Lazard correspondence}. Journal of Algebra \textbf{352}(2012), 430 - 450.
\bibitem[DK]{DK} Drutu C., Kapovich M., Geometric Group Theory, AMS Coll. Publ.
\textbf{63}(2018), 819 pp.
\bibitem[G]{G} Glukhov M.M., \textit{On applications of quasigroups in cryptography}. Applied Discrete Mathematics \textbf{2}(2008), 28-32.
 \bibitem[Gl]{Gl} Glushkov V.M., \textit{Locally nilpotent torsion free groups,
 complete over simple topological fields} Mat. Sb. \textbf{37(79)}(1955), 477-506.
 \bibitem[J]{J} Johnson, \textit{The group of formal power series
under substitution} J. Austral. Math. Soc. (series A) \textbf{45}(1988), 296 - 302. 
\bibitem[GK]{GK} Goze, M.; Khakimdjanov, Y., . Nilpotent and Solvable Lie Algebras. Handbook of Algebra, \textbf{2}(2000). Amsterdam: North-Holland, 615663.
 \bibitem[H]{H} Hirsch K.A., \textit{On infinite souble groups. II}. Proc. London Math. Soc. (2) \textbf{44}(1938), 336-344.
\bibitem[KM]{KM} Kargapolov, M. I.; Merzljakov, Ju. I., Fundamentals of the theory of groupsNauka, Moscow, 1972, 240 pp. (Russian)
Engl. Transl.: Graduate Texts in Mathematics, 62. Springer-Verlag, New York-Berlin, 1979. xvii+203 pp.
\bibitem[KHU]{KHU} Khukhro, E. I., $
p$-automorphisms of finite $p$-groups, LMS Lecture Notes Ser., v.246, Cambridge U. Press.,(1997),  xvii+204pp.
\bibitem[KS]{KS} Kochetova, J. V., Shirshova, E. E. \textit{On linearly ordered linear algebras} J. Math. Sci. \textbf{166}(2010), 725-732.
\bibitem[KK]{KK} Kokorin A.I., Kopytov V.M., Linerly ordered groups, Nauka, 1972, 200 pp. (in Russian).
\bibitem[KVM]{KVM}Kopytov, V. M., \textit{Ordering of Lie algebras}. Algebra Logika \textbf{11}(1972), 295-325.
\bibitem[K53]{K53} Kurosh, A. G.,  The theory of groups., 2d ed.: Gosizdat., Moscow, 1953, 467 pp. Engl. transl.:  Chelsea, New York, N.Y., vol 1 (1955), 272 pp., vol 2 (1956), 308 pp., 3rd ed. Nauka, Moscow, 1967, 648 pp.
\bibitem[Ku]{Ku} Kuzmin E.N., \textit{On a connection between Malcev algebras and
analytic Moufang loops}. Algebra i Logika  \textbf{10}(1971), 3-32.
\bibitem[Laz]{Laz} Lazard, M., \textit{Sur les groupes nilpotents et les anneaus de Lie}. Ann. sci. \'Ecole Norm. Sup. \textbf{71}(1954), 101-190.
\bibitem[VML]{VML}Levchuk, V. M., \textit{Connection between a unitriangle group and certain
rings, chap. 2, Groups of automorphisms}. Sibirian J. of Math. \textbf{24} (1983), 543-557.
\bibitem[L] {L} F.E.J. Linton, \textit{Coequalizors in categories of algebras}. Seminar on Triples and Categorical Homology Theory. Springer Lect. Notes Math.  \textbf{80} (1969), 75-90.
\bibitem[AM]{AM} Mahalanibis, A., \textit{The automorphism group of the group of unitriangle matrices over the field}.  Intern J. of Algebra, \textbf{7}(2013), 723-733.
\bibitem[AIM]{M}  Malcev, A. I., \textit{Torsion-free nilpotent groups}. Izv. Akad. Nauk SSSR, Ser. Mat., \textbf{13}(1949), 201 -212.
\bibitem[Ma]{Ma} Margulis G.A., Discrete subgroups of semisimple Lie groups, Springer-Verlag, Berlin (1991).
 \bibitem[MR]{MR}Myasnikov, A. G.; Remeslennikov, V. N. \textit{Degree groups. I. Foundations of the theory and tensor completions} (Russian) Sibirsk. Mat. Zh. \textbf{35} (1994), 1106 - 1118, iii; translation in Siberian Math. J. \textbf{35} (1994), 986 - 996 
\bibitem[MSPI1]{MSPI1}  Mostovoy, Jacob; P\'erez-Izquierdo, Jos\'e M.; Shestakov, Ivan P., \textit{On torsion-free nilpotent loops}. Q. J. Math. \textbf{70} (2019), 1091-1104.
\bibitem[Rem]{Rem} Remm, E., \textit{On filiform Lie algebras: Geometric and algebraic studies}. Rev. Romaine Math. Pures Appl. \textbf{63} (2018), 179 - 209.
\bibitem[SSS]{SSS} du Sautoy M.; Segal D.; Shalev A., \textit{New Horizons in pro-$p$-Groups}, Progress in Mathematics, \textbf{184}, Birkhauser (2000).
\bibitem[SHM]{SHM} Shmelkin, A.L., \textit{Complete nilpotent groups} (Russian), Algebra i Logika Sem. \textbf{6}(1967), 111 - 114.
\bibitem[SH]{SH} Shmelkin, A. L., \textit{Nilpotent products and nilpotent torsion-free groups} (Russian). Sibirsk. Mat. \v Z. \textbf{3} (1962),  625 - 640
\bibitem[Sm]{Sm} Smirnov D.M., \textit{Infrainvariant subgroups} (Russian). Scientific  Notes of Ivanovo Pedagogical University,
    \textbf{4}(1953), 82-96.
\bibitem[St]{St} Stewart I.N., \textit{An algebraic treatment of Malcev's theorems concerning nilpotent Lie groups and their Lie algebras}. Compositio Mathematica \textbf{22}, (1970), no.3, 289-312.

\end{thebibliography}
\end{document}